\documentclass[10pt,a4paper,reqno]{amsart}
\usepackage{amsmath,amssymb,amsthm,graphicx,color}
\usepackage[left=1in,right=1in,top=1in,bottom=1in]{geometry}
\usepackage{cite}
\usepackage[british]{babel}
\usepackage{enumerate}
\usepackage{etoolbox}
\makeatletter
\patchcmd{\@footnotetext}{\normalfont\footnotesize}{\normalfont\footnotesize}{}{}
\makeatother
\usepackage{hyperref} 
\usepackage{bm}
\usepackage{comment}
\usepackage{mathtools}
\usepackage{tikz}
\usetikzlibrary{calc,arrows}
\usepackage{fancyhdr}
\usepackage{booktabs}
\usepackage{tabularx}
\usepackage{multirow}
\usepackage{capt-of}
\newcolumntype{V}{>{\centering\arraybackslash\hsize=.37\hsize}X}
\newcolumntype{Z}{>{\centering\arraybackslash\hsize=1.72\hsize}X}
\newcolumntype{Y}{>{\centering\arraybackslash\hsize=0.91\hsize}X}

\usepackage{xcolor}
\definecolor{darkblue}{rgb}{0.11,0.11,0.52} 
\definecolor{darkred}{rgb}{0.94,0.28,0.28} 

\hypersetup{
    colorlinks=true,
    linkcolor=red, 
    citecolor=blue,
    urlcolor=green, 
    pdfborder={0 0 1},       
    linkbordercolor=red,     
    citebordercolor=green,   
    urlbordercolor=cyan      
}

\newtheorem{theorem}{Theorem}[section]

\newtheorem{proposition}[theorem]{Proposition}
\newtheorem{lemma}[theorem]{Lemma}

\numberwithin{equation}{section}

\theoremstyle{definition}

\theoremstyle{remark}
\newtheorem{remark}[theorem]{Remark}

\newtheorem*{remark*}{Remark}

\theoremstyle{definition}

 \allowdisplaybreaks

\newcommand{\1}[1]{{\mathbf 1}{_{\set*{#1}}}}

\newcommand{\N}{\mathbb{N}}
\newcommand{\Z}{\mathbb{Z}}

\newcommand{\R}{\mathbb{R}}

\newcommand{\B}{\mathbb{B}}

\newcommand{\E}{\mathbb{E}}
\renewcommand{\P}{\mathbb{P}}

\renewcommand{\AA}{\mathcal{A}}
\newcommand{\BB}{\mathcal{B}}

\newcommand{\DD}{\mathcal{D}}
\newcommand{\EE}{\mathcal{E}}
\newcommand{\FF}{\mathcal{F}}
\newcommand{\GG}{\mathcal{G}}

\newcommand{\LL}{\mathcal{L}}

\newcommand{\PP}{\mathcal{P}}

\newcommand{\UU}{\mathcal{U}}
\newcommand{\VV}{\mathcal{V}}
\newcommand{\WW}{\mathcal{W}}
\newcommand{\XX}{\mathcal{X}}

\newcommand{\eps}{\varepsilon}
\newcommand{\lstar}{\ell_*}
\newcommand{\lstarstar}{\ell_{**}}
\newcommand{\pstar}{p_\star}
\newcommand{\deltastar}{\delta_\star}
\DeclareMathOperator{\TV}{TV}
\DeclareMathOperator{\imatrix}{I}
\DeclareMathOperator{\id}{Id}
\DeclareMathOperator{\tr}{Tr}
\DeclareMathOperator{\Lip}{Lip}
\DeclareMathOperator{\diver}{div}
\newcommand{\dd}{\mathrm{d}}

\DeclarePairedDelimiter\abs\lvert\rvert
\DeclarePairedDelimiter\norm\lVert\rVert

\DeclarePairedDelimiter\inner\langle\rangle

\DeclarePairedDelimiter\paren{\lparen}{\rparen} 
\DeclarePairedDelimiter\bracket{\lbrack}{\rbrack}
\DeclarePairedDelimiter\set\lbrace\rbrace

\newcommand{\xeq}[2][]{\mathbin{\underset{#1}{\overset{#2}{=}}}}

\newcommand{\myand}{\quad \text{and} \quad}
\newcommand{\as}{\quad \text{as} \quad}
\newcommand{\myif}{\quad \text{if} \quad}
\providecommand{\for}{}
\renewcommand{\for}{\quad \text{for} \quad}
\newcommand{\myforall}{\quad \text{for all} \quad}
\newcommand{\mywhere}{\quad \text{where} \quad}
\newcommand{\myon}{\quad \text{on} \quad}

\begin{document}
\title[TLDs with polynomial-tailed targets: explosions, moment thresholds and ergodicity]{Tempered Langevin diffusions with polynomial-tailed targets: explosions, moment thresholds and ergodicity} 




\author{Aleksandar Mijatovi\'c}
\address{Department of Statistics, University of Warwick, UK}
\email{a.mijatovic@warwick.ac.uk}

\author{Haokai Xie}
\address{Department of Statistics, University of Warwick, UK}
\email{haokai.xie@warwick.ac.uk}


\begin{abstract}
    For tempered Langevin diffusions (TLDs) targeting polynomially tailed measures, the convergence rate is known to change from polynomial to exponential to uniform as the temperature rises through a critical value. 
    We explain this transition by the tails of the marginal laws of the norm of the TLD, up to a second, larger critical temperature beyond which the TLD explodes.
    As the temperature increases, these tails change from stretched-exponential, through super-polynomial at the first critical temperature, to polynomial, mirroring the transition in the convergence rate.
    Moreover, this tail analysis gives a complete quantification of the convergence rates throughout the non-explosive regime, in both $f$-variation and Wasserstein distances.

\end{abstract}

\maketitle

\section{Introduction} \label{sec:intro}

Langevin diffusions underpin many algorithms for sampling from probability distributions on $\R^d$ whose densities are known only up to a normalising constant~\cite{GrenanderMiller1994,roberts1996exponential,stramer1999langevin1,WellingTeh2011,GirolamiCalderhead2011,DurmusMoulines2017}.
For targets with polynomial tails, however, the classical Langevin diffusion fails to be exponentially ergodic~\cite{roberts1996exponential} and converges only at a polynomial rate~\cite{Veretennikov1997,FR2005,BM2026}.
Tempered Langevin diffusions (TLDs) were introduced to accelerate exploration of the tails by increasing the diffusion coefficient there, with a corresponding drift correction designed to preserve the target distribution~\cite{stramer1999langevin1,roberts2002langevin}.
For polynomially tailed targets, drift conditions show that increasing the temperature can improve convergence from polynomial to exponential and, at higher temperatures, to exponential convergence uniform over the initial state~\cite{FR2005,BM2026}.
These results quantify the improvement in convergence but leave open how it is reflected in the finite-time distributions of the TLD.
We explain this transition by characterising the tails of the TLD norm throughout the non-explosive regime, and justify a second, higher critical temperature beyond which the diffusion explodes in finite time almost surely.

Let $\Pi$ be a probability measure on $\R^d$, $d \in \Z_{>0}$, with Lebesgue density proportional to $\pi: \R^d \to (0, \infty)$. 
We assume throughout the paper that $\pi \in C^2(\R^d)$ has locally uniformly $\gamma$-H\"older continuous second-order partial derivatives for some $\gamma \in (0,1]$ and satisfies
\begin{equation} \label{AP}
    \tag{AP}
    \abs{x}^{d+r} \pi(x) \to c_\pi>0
    \myand
    \inner{x, \nabla \log \pi(x)} \to - (d+r) \as \abs{x} \to \infty 
\end{equation}
for some $r>0$. 
For $f: \R^d \to \R$, $f(x) \to c \in \R$ as $\abs{x} \to \infty$, means that $\lim_{\rho \to \infty} \sup_{\abs{x} = \rho} \abs{f(x) - c} = 0$. 
The probability measures with density proportional to $\pi$ satisfying \eqref{AP} have polynomial tails, with all moments of order \(p<r\) finite and all moments above $r$ infinite.

Given a temperature parameter $\ell \in [0,\infty)$, define the drift and diffusion coefficients
$$
b(x) = (1 - 2\ell ) \pi(x)^{-2 \ell } \nabla \log \pi(x) 
\myand 
\sigma(x) = \sqrt{2} \, \pi(x)^{-\ell }\imatrix_d, 
\quad x \in \R^d,
$$ 
where $\imatrix_d$ is the identity matrix in $\R^{d \times d}$, and consider the stochastic differential equation (SDE)
\begin{equation} \label{eq:tempered Langevin diffusion}
\dd X_t 
= b(X_t) \, \dd t + \sigma(X_t) \, \dd B_t
\myand
X_0=x\in\R^d, \tag{TLD}
\end{equation}
where $(B_t)_{t\ge0}$ is a standard Brownian motion in $\R^d$. 
Since $b$ and $\sigma$ are locally Lipschitz continuous, by \cite[Chap.~IV, Thms~2.3 \& 3.1]{IkedaWatanabe1989}, the SDE in~\eqref{eq:tempered Langevin diffusion} admits a unique strong solution $X$ up to the explosion time~$\zeta\in(0,\infty]$. 
We say $X$ is \textit{non-explosive} if $\zeta = \infty$ $\P_x$-a.s.\ for all $x \in \R^d$.
The solution $X$ is known as the \textit{tempered Langevin diffusion} process, which reduces to the classical Langevin diffusion when $\ell = 0$. 

By construction, the extended generator $\AA$ of $X$ (defined in Section~\ref{sec:proof explosion}) takes the divergence form $\AA f = \pi^{-1} \diver (\pi^{1-2 \ell} \nabla f)$ for $f \in C^2(\R^d)$. 
Hence, when $X$ is non-explosive, $\Pi$ is its unique invariant probability measure~\cite[Thms~4.1 \&~6.2]{Kent1978}.
It is known that the convergence of $X$ to $\Pi$ is polynomial for small $\ell$, including the classical Langevin diffusion ($\ell = 0$), and becomes exponential (resp.\ uniform) once $\ell$ reaches (resp.\ exceeds) the critical threshold $\frac{1}{d+r}$~\cite{stramer1999langevin1,FR2005,BM2026}.
\text{In this paper}, we first show that $X$ explodes in finite time once $\ell$ exceeds a second critical threshold~$\frac{1+r/2}{d+r}$, and then study the tails of $\abs{X_t}$ in the non-explosive regime.
These tails are the key to understanding why convergence speeds up as the temperature rises.
Moreover, our tail analysis yields a complete quantification of the convergence rates in both $f$-variation and Wasserstein distances.
Short YouTube presentations \cite{YouTube_talk} discuss our \href{https://youtu.be/P0yIqm6K5r0}{results} and their \href{https://youtu.be/ygnPmOUINWM}{proofs}, outlined in Subsections \ref{subsec:overview results} and \ref{subsec:intro discussion}, respectively.

\subsection{Overview of the main results} \label{subsec:overview results} 
Let us first define the $f$-variation distance and the Wasserstein distance on $\R^d$. 
For any measurable function $f: \R^d \to [1, \infty)$, the \textit{$f$-variation} of a signed measure $\mu$ on $(\R^d, \BB (\R^d))$ is $\norm{\mu}_f = \sup \set{ \abs{\int g \, \dd \mu}\text{: $g$ is measurable and } \abs{g} \le f}$. 
When $f \equiv 1$, $\norm{\mu}_f$ is the \textit{total variation} of $\mu$ denoted by $\norm{\cdot}_{\TV}$.\footnote{For any two probability measures $\mu, \nu$ on $(\R^d, \BB (\R^d))$, we have $\norm{\mu - \nu}_{\TV} = 2\sup_{A \in \BB(\R^d)} \abs{\mu(A) - \nu(A)}$.} 
For $p \in (0, \infty)$, let $\PP_p (\R^d)$ denote the set of probability measures on $(\R^d, \BB (\R^d))$ with finite $p$-th moment. 
We define the \textit{$L^p$-Wasserstein distance} for $\mu, \nu \in \PP_p (\R^d)$ by
$
\WW_p(\mu, \nu) = \inf_{\eta \in C(\mu, \nu)} \paren{\int_{\R^d \times \R^d} \abs{x-y}^p \, \eta (\dd x, \dd y) }^{{1}/ \paren{p \vee 1}},
$
where $C(\mu, \nu)$ collects all couplings between $\mu$ and $\nu$.\footnote{$\WW_p$ is a metric for $p \in (0,1)$, since $(A+B)^p \le A^p + B^p$ for all $A,B \ge 0$.}
We include $p \in (0,1)$ to cover the heavy-tailed case $r \in (0,1]$.

Our main results in Section~\ref{subsec:main results} characterise the explosions, moment thresholds and ergodic properties of $X$ are summarised in
Table~\ref{table:results}.  
The last column quantifies the rate of decay of $\norm{\P_x (X_t \in \cdot) - \Pi(\cdot)}_{f_p}$ for $f_p(x) = 1 + \abs{x}^p$ with $p \in [0,r)$, and of $\WW_p (\P_x (X_t \in \cdot), \Pi)^{p \vee 1}$ for $p \in (0,r)$, where the exact interpretation of these rates is given in Theorem~\ref{thm:TLD convergence rates}. 

\vspace{3pt}
\begin{center}
    \small
    \renewcommand{\arraystretch}{1}
    \setlength{\tabcolsep}{2.5pt}
    \renewcommand\tabularxcolumn[1]{>{\centering\arraybackslash}m{#1}}
    \begin{tabularx}{1.0 \textwidth}{@{}l V Z Y @{}}
        \toprule
        & \textbf{Explosion}
        & \textbf{Moment thresholds of $\abs{X_t}^p$ ($p >0$)} 
        & \footnotesize\textbf{$f_p$-variation ($p\in[0,r)$) \&\newline
        $L^p$-Wasserstein~$p\in (0,r)$} 
        \\
        \midrule
        $0\leq \ell<\frac{1}{d+r} $
        & $\zeta = \infty$ a.s.\
        & $\forall t > 0: \E_x \bracket{e^{\abs{X_t}^{p}}} < \infty \iff p \in \paren{0, 2(1-\ell(d+r))}$
        & $\approx t^{- \frac{r-p}{2 \paren*{1-\ell(d+r)}}}$ \\
        \midrule
        $\ell = \frac{1}{d+r}$
        & $\zeta = \infty$ a.s.\
        & $\forall t > 0: $ $\E_x \bracket*{\abs{X_t}^p} < \infty$, but $\E_x \bracket{e^{\abs{X_t}^{p}}} = \infty$
        & Exponential, \text{but not uniform} \\
        \midrule
        $\frac{1}{d+r} < \ell \le \frac{1+r/2}{d+r}$
        & $\zeta = \infty$ a.s.\
        & $ \E_x \bracket{\abs{X_t}^p} < \infty \iff p \in (0, r)$
        & Uniform \\
        \midrule
        $\frac{1+r/2}{d+r} < \ell < \infty$
        & $\zeta < \infty$ a.s.\
        & $\backslash$
        & $\backslash$ \\
        \bottomrule
    \end{tabularx}
    \captionof{table}{A summary of the main results (Theorems~\ref{thm:TLD explosion},~\ref{thm:TLD moments},~\ref{thm:TLD convergence rates}) on explosions, moments, and ergodicity of~\eqref{eq:tempered Langevin diffusion}  $X$ of temperature $\ell \in [0, \infty)$. 
    For $\frac{1}{d+r} < \ell \le \frac{1+r/2}{d+r}$, the moments are in fact bounded uniformly in space--time: $\sup_{s \ge t_0} \sup_{y \in \R^d} \E_y \bracket{\abs{X_s}^p} < \infty$ for all $t_0 > 0$ and $p \in [0,r)$.
    Ergodicity in the last column gives the rate of decay of $\norm{\P_x (X_t \in \cdot) - \Pi(\cdot)}_{f_p}$ for $f_p(x) = 1 + \abs{x}^p$ with $p \in [0,r)$, and of $\WW_p (\P_x (X_t \in \cdot), \Pi)^{p \vee 1}$ for $p \in (0,r)$. Some of the facts on explosions and $f$-variation in the table are from~\cite{FR2005,BM2026}, see Section~\ref{subsec:intro discussion} for details. }
    \label{table:results}
\end{center}
\vspace{3pt}

\subsection{Discussion of the results and our main contributions} \label{subsec:intro discussion}

We discuss the results in Table~\ref{table:results} from high- to low-temperature regimes, explaining the transitions in terms of the tail of $\abs{X_t}$ and outlining our main contributions along the way.

Under~\eqref{AP}, we prove that $X$ is non-explosive if and only if $\ell \le \frac{1+r/2}{d+r}$, and otherwise explodes in finite time almost surely from every starting point.
For targets with exactly polynomial tails outside a compact set,~\cite[Sec.~3.2]{FR2005} identified the same upper endpoint and established non-explosion there.
At the critical temperature, however, the radial Lyapunov construction underlying Khasminskii's criterion does not, in general, establish non-explosion under~\eqref{AP} alone (see Remark~\ref{remark:Khasminskii fails at lstarstar}).
To overcome this difficulty, we use the Itô trick~\cite{LyonsZhang1994}, exploiting the reversibility of reflected diffusions to control increments of logarithmic cutoff functions and thereby bound exit probabilities.
This yields non-explosion at the critical temperature under the asymptotic assumptions~\eqref{AP}, without requiring exact polynomial tails or additional assumptions on the rate of convergence in~\eqref{AP}.

For $\frac{1}{d+r} < \ell \le \frac{1+r/2}{d+r}$, we establish (to the best of our knowledge) the first uniform exponential convergence bounds for TLDs in $f_p$-variation and $L^p$-Wasserstein distances for every $p \in (0,r)$, where $f_p(x)=1+\abs{x}^p$.
Previously, uniform convergence was known only in total variation~\cite[Thm~16]{FR2005}.
The key ingredient in this extension is a bound on the moments that is uniform in both the initial state and time away from zero: for every $t_0>0$,
$$
    \sup_{s\ge t_0}\ \sup_{y\in\R^d}\
    \E_y\bracket{\abs{X_s}^p}<\infty
    \qquad\text{for every }p\in(0,r).
$$
Applying these estimates at a moment order strictly above $p$ allows us to upgrade uniform convergence in total variation to $f_p$-variation, from which the Wasserstein bounds follow.
To obtain the required moment estimates, we construct a time-dependent Lyapunov function using negative moments of Bessel processes, one of the main technical novelties of this paper; see Section~\ref{subsec:heuristics Lyapunov function} for the underlying intuition.

At $\ell = \frac{1}{d+r}$, we prove that $X$ is not uniformly ergodic. 
Using the lower bound on hitting-time probabilities in~\cite[Lemma~4.2]{BM2026}, we show that the total variation distance to~$\Pi$ at any time has supremum~$2$ over all starting points.
Bounding $\log \abs{X}$ by a Brownian motion with linear drift (see Section~\ref{subsubsec:heuristics moments}), we show that all polynomial moments of $\abs{X_t}$ are finite, whereas no positive power of $\abs{X_t}$ has finite exponential moments.
Hence, the tail of $\abs{X_t}$ is now thinner than that of~$\Pi$.
Accordingly, convergence in $f_p$-variation for $p \in [0,r)$ is merely exponential, as established in~\cite[Thm~16]{FR2005}.
We extend this exponential convergence to $L^p$-Wasserstein distances for $p \in (0,r)$.

For $0 \le \ell < \frac{1}{d+r}$, by relating a polynomial transformation of $\abs{X}$ to a Bessel process of negative dimension (see Section~\ref{subsubsec:heuristics moments}), we show that powers of $\abs{X_t}$ up to degree~$2(1 - \ell (d+r))$ have finite exponential moments.
We also prove matching polynomial upper and lower bounds on the convergence rate in $L^p$-Wasserstein distances for $p \in (0,r)$, which were previously known only in $f_p$-variation~\cite[Thm~16]{FR2005} and~\cite[Thm~3.1]{BM2026}.
This slow polynomial rate is again explained by the tail of $\abs{X_t}$, which is now even thinner than at $\ell = \frac{1}{d+r}$.
Indeed, the closer the tail of $\abs{X_t}$ is to that of~$\Pi$, the faster the convergence: as $\ell$ decreases from $\frac{1}{d+r}$ to~$0$, the largest power of $\abs{X_t}$ with finite exponential moments increases, while the degree of the polynomial rate decreases.

\subsection{Prior work on sampling from heavy-tailed targets} \label{subsec:related literature}

The study of TLDs with polynomially decaying targets goes back to Stramer and Tweedie \cite[Thm~3.1]{stramer1999langevin1}, who showed in dimension one that convergence is not exponential below the critical temperature $\frac{1}{1+r}$.
Roberts and Stramer \cite[Thm~3.3]{roberts2002langevin} gave drift conditions for exponential ergodicity in $\R^d$, showing in particular that for multivariate Student-$t$ targets the TLD at the critical temperature $\frac{1}{d+r}$ is exponentially ergodic, whereas the classical Langevin diffusion is not.
Fort and Roberts \cite[Thm~16]{FR2005} then obtained upper bounds on the convergence rate in $\R^d$ over the whole non-explosive temperature range, and Bre{\v{s}}ar and Mijatovi{\'{c}} \cite[Thm~3.1]{BM2026} later showed that the polynomial bound below the critical temperature is sharp.
Douc, Fort and Guillin \cite[Thm~5.5]{DoucSubgeometric09} derived upper bounds analogous to \cite[Thm~16]{FR2005} for TLDs whose targets have subexponentially decaying tails.
We expect the tail of the TLD norm to explain the convergence rates in that setting as well, and leave this question for future research.

Weighted Poincaré inequalities~\cite{BobkovLedoux2009} offer another route to exponential ergodicity at $\ell = \frac{1}{d+r}$.
Assuming~\eqref{AP} and that the function $x \mapsto \pi(x)^{-1/(d+r)}$ is convex, $\Pi$ satisfies a weighted Poincaré inequality with weight $1 + \abs{x}^2$~\cite[Prop.~3.2]{CGGR2010}, in the sense of \cite[Def.~2.4]{CGGR2010}.
This is then essentially the Poincaré inequality for $X$ at $\ell = \frac{1}{d+r}$, as observed in \cite[Sec.~2.1]{MN2026}.
Hence the semigroup of $X$ decays exponentially fast in $L^2(\Pi)$ on $\Pi$-mean-zero functions.
Since $X$ is $\Pi$-reversible (see Proposition~\ref{prop:TLD density}), this decay carries over to total variation distance whenever the density of the initial distribution with respect to $\Pi$ is in $L^2(\Pi)$.
Unlike the drift-condition approach, this yields an explicit lower bound on the exponential rate.
Moreover, $\frac{1}{d+r}$ is the smallest temperature at which this Poincaré inequality holds, and the same argument gives exponential convergence for $\frac{1}{d+r} < \ell \le \frac{1+r/2}{d+r}$. 

Related acceleration strategies have also been developed for piecewise deterministic Markov processes (PDMPs).
\cite{VasdekisRoberts2023} introduced the Speed Up Zig-Zag process, using a position-dependent speed to improve convergence for heavy-tailed targets.
\cite{BertazziVasdekis2025} placed such constructions within a general framework of time-changed Markov processes, covering both diffusions and PDMPs.
A different construction is the tempered Bouncy Particle Sampler in~\cite{MN2026}, which retains
piecewise-constant velocities and uses position-dependent velocity refreshments.
Its explicit exponential convergence bounds are obtained through hypocoercivity and a weighted Poincaré inequality for the associated tempered Langevin diffusion.

\subsection{Structure of the rest of the paper}
In Section~\ref{sec:main results and heuristics}, we present our main results in Table~\ref{table:results} and some heuristics behind them. The proofs of explosion, moment properties and ergodicity are contained in Sections~\ref{sec:proof explosion},~\ref{sec:proof moments} and~\ref{sec:proof ergodicity}, respectively.

\section{Main results} \label{sec:main results and heuristics}

Having summarised our main results informally in Table~\ref{table:results}, we now state them precisely in Section~\ref{subsec:main results} and give the probabilistic intuition behind them in Section~\ref{subsec:heuristics}. 
We then outline the construction of a time-dependent Lyapunov function in Section~\ref{subsec:heuristics Lyapunov function}, which addresses one of the main technical contributions. 
Throughout the rest of this paper, we write 
$$
\lstar \coloneqq \frac{1}{d+r} 
\myand 
\lstarstar \coloneqq \frac{1+r/2}{d+r},
$$ 
for the critical temperatures for convergence rates and explosion, respectively.

\subsection{Main results} \label{subsec:main results} 
We begin by studying the explosion time.
\begin{theorem} \label{thm:TLD explosion}
    $X$ is non-explosive if and only if $0 \le \ell \le \lstarstar$.
    Moreover, if $\ell > \lstarstar$, then $\zeta < \infty$ $\P_x$-a.s.\ for all $x\in\R^d$.
    
\end{theorem}

We focus on the case that $X$ is non-explosive from now on. Let $(\P_x)_{x\in\R^d}$ be a family of probability measures under which $X$ is a Feller process. 
Throughout, $\E_x$ denotes the expectation under~$\P_x$. 

\begin{theorem} \label{thm:TLD moments}
    $ $
    \begin{enumerate}
        \itemsep 5pt
        \item If $\ell  \in [0, \lstar)$, then for every $x \in \R^d$ and $t >0$ we have
        \begin{equation*}
            \E_x \bracket*{\exp \paren*{\abs{X_t}^p}}
            \begin{cases}
                 < \infty, & \myif p \in \left[ 0, 2 (1- \ell(d+r)) \right), \\
                = \infty, & \myif p \in (2 (1- \ell(d+r)) , \infty).
            \end{cases}
        \end{equation*}
        For $p = 2 (1- \ell(d+r)) $, there exist some constants $0 < t_1 < t_2 < \infty$ such that $\E_x \bracket*{e^{\abs{X_t}^p}} < \infty$ whenever $t < t_1$, and $\E_x \bracket*{e^{\abs{X_t}^p}} = \infty$ for every $t > t_2$.
        \item If $\ell = \lstar$, then for every $x \in \R^d$ and $t >0$ we have
        \begin{equation} \label{eq:all poly moments at lstar}
            \E_x \bracket{\abs{X_t}^p} < \infty 
            \myand 
            \E_x \bracket*{\exp \paren*{\abs{X_t}^p}} = \infty
            \myforall 
            p \in (0, \infty).
        \end{equation}
        \item If $\ell  \in (\lstar, \lstarstar]$, then for every $t_0 >0$,
        \begin{equation} \label{eq:moments uniform in time and space above lstar}
            \sup_{s \ge t_0}\ \sup_{y \in \R^d}\ \E_y \bracket{\abs{X_s}^p} < \infty  \myforall p \in [0, r).
        \end{equation}
        Conversely, if $\E_x \bracket{\abs{X_t}^p} <\infty$ for some $x\in\R^d$ and $t>0$, then $p \in [0,r)$.
    \end{enumerate}
\end{theorem}

The following theorem quantifies the rate at which the law of $X_t$ converges to $\Pi$. 

\begin{theorem} \label{thm:TLD convergence rates}
    Let $f_p (x) = 1 + \abs{x}^p$ for $x \in \R^d$, $p \ge 0$. 
    \vspace{3pt}
    \begin{enumerate}
        \itemsep 5pt
        \item If $\ell  \in [0, \lstar)$, $X$ fails to be exponentially ergodic. More precisely, fix $p \in (0,r)$ and set $\beta_p := (r-p)/\paren*{2 (1- \ell(d+r))} > 0$. For any $\eps > 0$ and $x \in \R^d$, there exist constants $C_{-}(x), C_{+}(x) > 0$ such that, for all $t \ge 1$,  
        \begin{equation} \label{ergodicity:polynomial Wasserstein}
            C_{-}(x) t^{- \beta_p - \eps} 
            \le 
            \WW_p (\P_x (X_t \in \cdot), \Pi)^{p \vee 1}
            \le 
            C_{+}(x) t^{- \beta_p + \eps}. 
        \end{equation}
        \item If $\ell = \lstar$, $X$ is exponentially ergodic. 
        For every $p \in (0, r)$, there exists a constant $c >0$ such that, for every $x \in \R^d$, there is a constant $C(x)>0$ satisfying
        \begin{equation} \label{ergodic:exponentially Wasserstein}
            \WW_p (\P_x (X_t \in \cdot), \Pi)^{p \vee 1}
            \le 
            C(x) e^{- c t} \quad \text{for all} \quad t \ge 1.
        \end{equation}
        However, $X$ is not uniformly ergodic as $\sup_{x \in \R^d} \norm{\P_x (X_t \in \cdot) - \Pi(\cdot)}_{\TV} = 2$ for all $t > 0$. 
        \item If $\ell \in (\lstar, \lstarstar]$, $X$ is uniformly ergodic. 
        For every $p \in (0, r)$ there exist constants $c, C>0$ such that, for all $x \in \R^d$,
        \begin{equation} \label{ergodic:uniformly f var}
            \norm{\P_x (X_t \in \cdot) - \Pi(\cdot)}_{f_p} 
            \le 
            C e^{- c t} \quad \text{for all} \quad t \ge 1.
        \end{equation}
        Moreover, for every $p \in (0,r)$, this bound \eqref{ergodic:uniformly f var} remains valid with the $f_p$-variation distance replaced by ${\WW_p (\P_x (X_t \in \cdot), \Pi)^{p \vee 1}}$. 
    \end{enumerate}
\end{theorem}

\begin{remark} \label{remark:convergence rate results}
    As noted in Section~\ref{subsec:intro discussion}, parts of Theorem~\ref{thm:TLD convergence rates} were already known in $f_p$-variation.
    The upper bounds in \eqref{ergodicity:polynomial Wasserstein} and \eqref{ergodic:exponentially Wasserstein} for $p \in [0,r)$, as well as \eqref{ergodic:uniformly f var} in total variation, were obtained in \cite[Thm~16]{FR2005}, under a related but different assumption on $\pi$ and for $f$-variations close to $f_p$.
    The matching lower bound in \eqref{ergodicity:polynomial Wasserstein} for $d \ge 2$ follows directly from \cite[Thm~3.1]{BM2026}.
    Beyond these results, we extend all of the above bounds to the $L^p$-Wasserstein distances, prove that the exponential convergence at $\ell = \lstar$ is not uniform, and upgrade \eqref{ergodic:uniformly f var} from total variation to $f_p$-variation for every $p \in (0,r)$.
\end{remark}

\subsection{Probabilistic intuition} \label{subsec:heuristics}
We now give some intuitive justification for our main results. 
For simplicity, we assume that $\pi$ has an exact polynomial tail, that is, $\pi(x) = \abs{x}^{-d-r}$ for $\abs{x}$ sufficiently large. 
We first explain heuristically why the critical temperature thresholds $0<\lstar<\lstarstar$ arise and then provide an intuitive explanation for the behaviour of the moments of $\abs{X_t}$ in the various regimes via Bessel processes. 

\subsubsection{Critical temperature thresholds}\label{subsubsec:critical temperature thresholds}
Both critical temperature thresholds, the one at which the convergence rate changes and the one at which explosion sets in, can be read off by applying the extended generator to powers of the norm.
Let $\gamma > 0$ and let $\VV_\gamma \in C^2(\R^d)$ be a function with values in $[1,\infty)$ such that $\VV_\gamma (x) = \abs{x}^\gamma$ for all sufficiently large $\abs{x}$.
Applying the extended generator $\AA$, defined in Section~\ref{sec:proof explosion}, to $\VV_\gamma$ gives, for all sufficiently large $\abs{x}$,
\begin{equation}\label{eq:exact poly tail generator on polynomials}
    \AA \VV_\gamma 
    = \inner{b, \nabla \VV_\gamma} + \pi^{-2\ell} \Delta \VV_\gamma
    = - \gamma \paren*{(r - 2 \alpha) - \gamma} \VV_\gamma^{1 + {2 \alpha}/{\gamma}}, 
    \quad \text{with} \quad
    \alpha := \ell (d+r) - 1. 
\end{equation} 

If $0 \le \ell < \lstarstar$, then $r - 2\alpha > 0$, and for every $\gamma \in (0,r - 2 \alpha)$ the right-hand side of \eqref{eq:exact poly tail generator on polynomials} is negative; hence \eqref{eq:exact poly tail generator on polynomials} is a Foster--Lyapunov drift condition outside a compact set. As $\ell$ increases through $\lstar$, the parameter $\alpha$ changes sign, and the function $u \mapsto u^{1 + {2 \alpha}/{\gamma}}$ passes from sublinear through linear to superlinear; by \cite{DoucSubgeometric09,dmt1995exponential}, these three cases correspond to polynomial, exponential and uniform convergence, respectively.

In contrast, if $\ell > \lstarstar$, or equivalently $2 \alpha > r$, we have the constant $- \gamma \paren*{(r - 2 \alpha) - \gamma} > 0$ and the exponent $1 + {2\alpha}/{\gamma} > 1$ for every $\gamma > 0$. 
Thus, in this regime, $\VV_\gamma(X)$ has a positive superlinear drift and therefore explodes in finite time, just as a positive function $f$ on $\R_+$ satisfying $f' = c f^{\beta}$ with $c>0$ and $\beta > 1$ blows up in finite time.

At the critical temperature $\ell = \lstarstar$, $X$ is non-explosive, even though $\AA \VV_\gamma = \gamma^2 \VV_\gamma^{1 + {r}/{\gamma}} > 0$ for every $\gamma > 0$ and all sufficiently large $\abs{x}$.
One way to see this is that $\AA \log \abs{x} = (2\alpha - r) \abs{x}^{2\alpha} = 0$ for all sufficiently large $\abs{x}$ in this regime.
This also indicates that, compared with the case $0 \le \ell < \lstarstar$, proving non-explosion requires a more delicate argument, for which we use the Itô trick \cite{LyonsZhang1994} as discussed in Section~\ref{subsec:intro discussion}.

\subsubsection{Moment thresholds} \label{subsubsec:heuristics moments}
The moment properties can be understood through a spatial transformation that turns $\abs{X_t}$ into a Bessel process.
Let $\pi(x) = \abs{x}^{-d-r}$ whenever $\abs{x} \ge R$ for some $R>0$. 
Let $\Phi: \R^d \to (0, \infty)$ be a $C^2$ function such that: for $\abs{x} \ge R$, $\Phi(x) = \abs{x}^{- \alpha} / \paren*{\sqrt{2} \abs{\alpha}}$ when $\ell \neq \lstar$ and $\Phi (x) = \log \abs{x} / \sqrt{2}$ when $\ell = \lstar$. 
Fix $x_0 \in \R^d$ with $\abs{x_0} > R$.
With $\tau_R = \inf \set*{t>0: \abs{X_t} \le R}$, Itô's formula gives, on $\set{t \le \tau_R}$,
\begin{equation} \label{eq:heuristic spatial transformation}
    \Phi(X_t)
= 
\Phi(x_0) + \operatorname{sgn}(\alpha) \frac{r - \alpha}{\sqrt 2} \int_0^t  \abs{X_s}^{\alpha} \dd s + W_t
\xeq{\text{if $\ell \neq \lstar$}}
\Phi(x_0) + \int_0^t \frac{\deltastar - 1}{2 \Phi(X_s)} \dd s + W_t
\end{equation}
where $W$ is a standard one-dimensional Brownian motion, $\operatorname{sgn}(\alpha)$ is the sign of $\alpha$, with the convention $\operatorname{sgn}(0) = -1$, and $\deltastar = r/\alpha$ for $\ell \neq \lstar$.

For $0 \le \ell \le \lstar$, we have $\alpha \le 0$, so the drift of $\Phi(X)$ is negative.
Thus, up to time $\tau_R$, $\Phi(X)$ is dominated by a Brownian motion, and in this regime the moment properties of $\abs{X}$ are inherited from the exponential moments of Brownian motion.

For $\lstar < \ell \le \lstarstar$, $\Phi(X)$ behaves like a Bessel process of dimension $\deltastar \in [2,\infty)$ up to time $\tau_R$.
Rewriting $\abs{X_t}^p$ in terms of $\Phi(X_t)$ gives, for $p \in [0,r)$,
\begin{equation}\label{heuristic:neg moments Bessel}
    \E_{x_0} \bracket*{\abs{X_t}^p \1{t < \tau_R}}
    =
    \paren{\sqrt{2} \alpha}^{- \nu} \E_{x_0} \bracket*{\Phi(X_t)^{-\nu} \1{t < \tau_R}} \for
    \nu = p/\alpha \in [0, \deltastar),
\end{equation}
which explains the equivalence $\E_x \bracket*{\abs{X_t}^p} < \infty \iff p \in [0,r)$: a negative moment of a Bessel process is finite if and only if its order is strictly smaller than the dimension.
Moreover, these negative moments stay finite even when the Bessel process starts at zero, which corresponds to $\abs{X}$ starting at ``infinity''.
We therefore expect the upper bound on the $p$-th moment of $\abs{X_t}$ to be uniform in space and, by the Markov property, also uniform in time, as confirmed by Theorem~\ref{thm:TLD moments}~(3).

This connection to Bessel processes also explains why $X$ explodes in finite time almost surely for $\ell > \lstarstar$. 
Indeed, a Bessel process of positive dimension hits zero almost surely if and only if its dimension lies in $(0,2)$, and $\deltastar = r/\alpha \in (0,2)$ holds exactly in this regime.

\subsection{A time-dependent Lyapunov function} \label{subsec:heuristics Lyapunov function}

We now outline how to prove that $\abs{X_t}$ has finite $p$-th moments for $p \in (0,r)$ when $\lstar < \ell \le \lstarstar$, which is one of the main technical contributions of this work. 
For general $\pi$ satisfying~\eqref{AP}, the connection with Bessel processes in~\eqref{eq:heuristic spatial transformation} does not yield this bound directly.
The path domination argument used for $0 \le \ell \le \lstar$ breaks down in this regime, because the infimum of a Bessel process has fewer finite negative moments than the process itself (see \eqref{eq:Bessel lose moments} in Section~\ref{subsec:moments Bessel}).

We instead use a Grönwall-type inequality for expectations~\cite[Lemma~2.10]{BM2026}.
Fix $p \in (0,r)$ and let $\AA$ be the extended generator of $X$, as defined in Section~\ref{sec:proof explosion}.
By this inequality, to show finite $p$-th moments it suffices to find a function $\UU$ that approximates $\abs{x}^p$ for large $\abs{x}$ and satisfies the drift condition $\AA \, \UU \le c \, \UU$ for some $c \ge 0$.
The obvious choice $\UU(x) = \abs{x}^p$ fails when $p \in (r - 2\alpha, r)$, because $\AA \abs{x}^p \ge c' \abs{x}^{p + 2\alpha}$ for some constants $c' > 0$ (recall $\alpha > 0$) whenever $\abs{x}$ is sufficiently large (see \eqref{eq:generator on polynomial} in Section~\ref{sec:proof moments}). 

We resolve this issue by letting $\UU$ depend on time and satisfy $(\partial_t + \AA) \, \UU \le c \, \UU$.
As noted in~\eqref{heuristic:neg moments Bessel}, a natural candidate is the closed form of the $\nu$-th negative moment of a Bessel process of dimension $\deltastar=r/\alpha$, which we denote by $U_{\deltastar, \nu}$ (defined in \eqref{eq:closed forms of Bessel negative moments} in Section~\ref{subsec:moments Bessel}). 
Fix $T>0$ and set $\UU(t,x) = U_{\deltastar, \nu} (T-t, \Phi(x))$ for $t \in [0,T]$.
Since $\Phi(X)$ evolves as a Bessel process of dimension $\deltastar$ away from the origin, we expect $U_{\deltastar, \nu} (s, \Phi(x))$ to solve the Kolmogorov backward equation $\partial_s f = \AA f$, which would imply $(\partial_t + \AA) \, \UU = 0$.
It turns out that, after a small modification, $\UU$ indeed approximates $\abs{x}^p$ for large $\abs{x}$ and satisfies $(\partial_t + \AA) \, \UU \le 0$ as shown in Proposition~\ref{prop:Bessel Lyapunov is non-positive}.

\section{Explosions} \label{sec:proof explosion}

We prove the explosion results in Theorem~\ref{thm:TLD explosion} in this section. 
Let us first introduce some notation. 
For $\rho \ge 0$, let $\B_\rho = \set{x \in \R^d: \abs{x}< \rho}$ denote the open ball of radius $\rho$, with closure $\bar{\B}_\rho$.
For a subset $A \subset \R^d$, $A^c$ denotes its complement. 
Throughout the rest of this paper, we assume without loss of generality that $c_\pi = 1$. 

We then define the extended generator of $X$. 
Following \cite[Chap.~VII, Def.~1.8]{RevuzYor1999}, let $\DD (\AA)$ be the set containing all measurable functions $V: \R^d \to \R$ such that: there exists a measurable function $W: \R^d \to \R$ such that for every $t \ge 0$ and $x \in \R^d$, $\int_0^t \abs{W(X_s)} \dd s < \infty$ $\P_x$-almost surely and 
\begin{equation} \label{eq:extended generator definition}
    \paren*{V(X_t) - V(x) - \int_0^t W(X_s)  \dd s}_{t \ge 0} 
    \quad
    \text{is a $\P_x$-local martingale w.r.t.\ $(\FF_t)_{t \ge 0}$}.  
\end{equation}
We then denote $W = \AA V$ and call $(\AA, \DD(\AA))$ the \textit{extended generator} of $X$. 
By Itô's formula, we have: for any $V \in C^2(\R^d)$, 
\begin{equation} \label{eq:generator on general functions}
    \AA V
    = \inner{b, \nabla V} + \frac{1}{2} \tr (2 a \nabla^2 V)
    = \pi^{-1} \diver (\pi a \nabla V)
    = \pi^{-2\ell } \bracket*{(1 - 2 \ell) \inner{\nabla \log \pi, \nabla V} + \Delta V }, \quad 
    a := \pi^{-2 \ell}\imatrix_d,
\end{equation}
where $\nabla^2$ is the Hessian and $\diver$ is the divergence operator.

\subsection{Explosion time properties for \texorpdfstring{$\ell \neq \lstarstar$}{ℓ ≠ ℓ**}} \label{subsec:explosion Khasminskii}

We study the explosion time via the following \textit{Khasminskii's explosion test} criteria. 

\begin{lemma} \label{lemma:explosion test}
    Define the localising sequence $\zeta_n = \inf \set{t>0: \abs{X_t} \ge n}$, $n \in \N$.  
    \begin{enumerate}
        \item If there exists a function $V \in C^2 (\R^d)$ such that: for some $K>0$, $\AA V \le K V$ on $\R^d$ and $\lim_{\rho \to \infty} \inf_{\abs{x} = \rho} V(x) = \infty$, then $X$ is non-explosive; 
        \item If there exists a bounded function $V \in C^2 (\R^d)$ such that $\AA V \ge c >0$ on $\R^d$, then: for every $x$ in $\R^d$, $\E_x[\zeta] \le \frac{1}{c} \paren*{\sup_{y \in \R^d} V(y) - V(x)}$; in particular, $\zeta < \infty$ $\P_x$-a.s.
    \end{enumerate}
\end{lemma}

\begin{proof}
    The assertion in~(1) follows from~\cite[Thm~3.5]{Khasminskii2012}. For~(2), 
    assume $\AA V \ge c >0$ on $\R^d$ for some bounded function $V$ and denote $\norm{V}_\infty\coloneqq \sup_{y \in \R^d} \abs{V(y)} <\infty$.
    For $x \in \B_n$ and $t>0$, applying Itô's formula on the stochastic interval  $[0,\zeta_n]$ gives 
    $$
    \norm{V}_\infty
    \ge 
    \E_x V(X_{t \wedge \zeta_n})
    = 
    V(x) + \E_x \int_0^{t \wedge \zeta_n} \AA V(X_s)\dd s
    \ge
    V(x)+c \E_x \bracket*{t \wedge \zeta_n}.
    $$
    Since 
    $\lim_{t \to \infty} \lim_{n \to \infty} (t \wedge \zeta_n)=\zeta$ a.s., we conclude by the monotone convergence theorem.
\end{proof}

To construct suitable Lyapunov functions to apply Lemma~\ref{lemma:explosion test}, we consider a radial function $U(x) = f(\abs{x})$ where $f \in C^2((0,\infty))$. 
Using \eqref{eq:generator on general functions}, one computes that, for $x \neq 0$,
$$
\AA U(x) = \pi(x)^{-2 \ell} \set*{f''(\abs{x}) + f'(\abs{x}) \psi(x)}
\mywhere \psi(x) = \frac{d-1}{\abs{x}} + (1-2\ell) \frac{\inner{\nabla \log \pi(x), x}}{\abs{x}}.
$$
Define
$$
\psi_+ (s) = \frac{d-1}{s} + \sup_{\abs{x}=s} (1-2\ell) \frac{\inner{\nabla \log \pi(x), x}}{s}, \qquad s>0,
$$
and let $\psi_-(s)$ denote the same expression with the supremum replaced by the infimum. 
If $f' \ge 0$, then 
$$
\pi(x)^{-2 \ell} \set*{f''(\abs{x}) + f'(\abs{x}) \psi_- (\abs{x})}
\le
\AA U(x)
\le 
\pi(x)^{-2 \ell} \set*{f''(\abs{x}) + f'(\abs{x}) \psi_+ (\abs{x})} 
$$
for $x \neq 0$. 
Set, for $t >0$ and $s \ge 0$, 
$$
\Psi_{\pm}(t) = \int_1^t \psi_{\pm}(s)\,\dd s 
\myand 
h(s) = \sup_{\abs{y}=s}\pi(y)^{2\ell}. 
$$

Let $\nu = d - 1 - (1-2\ell) (d+r)$ and then $s \psi_\pm (s) \to \nu$ as $s \to \infty$ by \eqref{AP}. 

\begin{proof}[Proof of Theorem~\ref{thm:TLD explosion} \textup{($\ell \in [0, \lstarstar)$)}]
    Following \cite[p.~1578]{FR2005}, consider $U_+: \R^d \to (0,\infty)$ in $C^2(\R^d)$ with 
    \begin{equation} \label{Lyapunov:non explosion}
        U_+(x) = f_+(\abs{x}) \for x \in \bar{\B}_1^c, 
        \mywhere
        f_+(\rho) = 1 + \int_1^\rho \exp \paren*{- \Psi_+ (s)} \dd s \for \rho \ge 1. 
        \end{equation}
    Since $f_+' (\rho) = \exp \paren*{- \Psi_+(\rho)} \ge 0$ and $f_+'' (\rho) = -\psi_+ (\rho) \exp \paren*{- \Psi_+(\rho)}$, we obtain
    $$
    \AA U_+(x) = \pi (x)^{-2\ell} \exp \paren*{- \Psi_+(\abs{x})} \bracket*{\psi(x) - \psi_+(\abs{x})} \le 0 
    \for 
    x \in \bar{\B}_1^c.
    $$
    Thus, there exists $K>0$ such that $\AA U_+ \le K U_+$ on $\R^d$, as $\AA U_+$ is continuous and $U_+$ is positive on the compact set $\bar{\B}_1$. 
    
    If $\ell \in [0, \lstarstar)$, equivalently $\nu < 1$, fix $\tilde{\nu} \in (\nu, 1)$; then there exists $s_0 > 1$ such that $s \psi_+(s) < \tilde{\nu}$ for $s \ge s_0$. 
    Consequently,
    $
    \Psi_+ (t) - \Psi_+ (s_0)
    \le \int_{s_0}^t \tilde{\nu} s^{-1} \dd s
    = \tilde{\nu} \log (t/s_0)
    $
    and for some constant $c>0$,
    $$
    f_+ (\rho)
    \ge \int_{s_0}^\rho \exp \paren*{- \Psi_+(s_0) - \tilde{\nu} \log \frac{s}{s_0}} \dd s
    = c \int_{s_0}^\rho s^{-\tilde{\nu}} \dd s. 
    $$
    Since $\tilde{\nu} < 1$, we obtain $\lim_{\rho \to \infty} \inf_{\abs{x} = \rho} U_+(x) = \lim_{\rho \to \infty} f_+ (\rho) = \infty$. 

    By Lemma~\ref{lemma:explosion test} (1), we conclude that $X$ is non-explosive if $\ell \in \left[0, \lstarstar\right)$. 
\end{proof}

\begin{remark} \label{remark:Khasminskii fails at lstarstar}
    When $\ell = \lstarstar$ and $d \ge 3$, \eqref{AP} alone does not guarantee that $\lim_{\rho \to \infty} \inf_{\abs{x} = \rho} U_+(x) = \infty$. 
    To see this, set $\vartheta(s) = \sup_{\abs{x} = s} \set*{ (1-2\ell) \paren*{\inner{\nabla \log \pi (x), x} + (d+r)}}$ for $s \ge 1$, so that $\Psi_+ (t) = \log t + \int_1^t \vartheta(s) s^{-1} \dd s$.
    For instance, if $\vartheta(s) = 2 / \log s$ for $s \ge e$, then $\Psi_+(t) = \Psi_+(e) - 1 + \log t + 2 \log \log t$ for $t \ge e$, and hence $f_+(\rho) = f_+(e) + e^{1 - \Psi_+(e)} \int_e^\rho t^{-1} (\log t)^{-2} \dd t \le f_+(e) + e^{1 - \Psi_+(e)}$ for all $\rho \ge e$, so $f_+$ stays bounded.
    To conclude that $\lim_{\rho \to \infty} \inf_{\abs{x} = \rho} U_+(x) = \infty$, we therefore need to assume, in addition to~\eqref{AP}, that $\vartheta(s)$ converges to $0$ sufficiently fast as $s \to \infty$.
    If $\vartheta(s)$ is integrable with respect to $s^{-1} \dd s$, then there exists $C>0$ such that $\Psi_+ (t) \le \log t + C$ for $t \ge 1$.
    Thus, since $f_+(\rho) \ge e^{-C}\int_1^\rho t^{-1} \dd t=e^{-C}\log \rho$, we have $\inf_{\abs{x} = \rho} U_+(x) = f_+ (\rho) \to \infty$ as $\rho \to \infty$. 
\end{remark}

\begin{proof}[Proof of Theorem~\ref{thm:TLD explosion} \textup{($\ell \in (\lstarstar, \infty)$)}]
    Assume $\ell \in (\lstarstar, \infty)$, equivalently $\nu > 1$.
    Define
    $$
    \varphi(\rho) = e^{- \Psi_- (\rho)} \int_0^\rho e^{\Psi_- (t)} h(t) \dd t
    \myand
    f_-(\rho) = \int_0^\rho \varphi(t) \dd t,
    $$
    for $\rho > 0$ and set $U_-(x) = f_- (\abs{x})$ for $x \neq 0$.

    \smallskip
    \noindent  
    \underline{$U_-$ extends to $C^2 ( \R^d )$.}
    For $\rho \in (0,\infty)$, one computes $\varphi'(\rho) = - \psi_-(\rho) \varphi(\rho) + h(\rho)$.
    Since $\nabla \log \pi$ is continuous, $\psi_- (\rho) = \frac{d-1}{\rho} + O(1)$ as $\rho \downarrow 0$, so $\xi(\rho) = \int_1^\rho \paren*{\psi_-(s) - \frac{d-1}{s}} \dd s$ extends to a continuous function $\xi: [0,1] \to \R$ and $e^{\Psi_- (\rho)} = \rho^{d-1} e^{\xi (\rho)}$ for $\rho \in (0,1]$.
    Since 
    \begin{align*}
        \varphi(\rho)
        = \rho^{1-d} e^{- \xi (\rho)} \int_0^\rho t^{d-1} e^{\xi(t)} h(t) \dd t
        \sim \rho^{1-d} e^{- \xi (0)} e^{\xi (0)} h(0) \frac{\rho^d}{d}
        \sim h(0) \frac{\rho}{d}
    \end{align*}
    as $\rho \downarrow 0$, it follows that $\varphi'(\rho) = - \paren{\frac{d-1}{\rho} + O(1)} h(0) \frac{\rho}{d} (1 + o(1))  + h(\rho) \to \frac{h(0)}{d}$ as $\rho \downarrow 0$. 
    Therefore, $f_-$ extends to $C^2 ([0, \infty))$ with 
    \begin{equation} \label{eq:f_- at zero}
        f_-(0) = 0, 
        \quad
        f_-'(0) = \varphi(0) = 0, 
        \myand
        f_-''(0) = \varphi'(0) = h(0)/d.
    \end{equation}
    Since $f_- \in C^2 ([0, \infty))$ and $f_-'(0) = 0$, we conclude via \eqref{eq:f_- at zero} that $U_- \in C^2(\R^d)$ with $U_-(0)=0$, $\nabla U_-(0) = 0$ and $\nabla^2 U_-(0) = (h(0)/d) \imatrix_d$.

    \smallskip
    \noindent
    \underline{$U_-$ is bounded.}
    Since $t \psi_-(t) \to \nu$ and $h(t) \sim t^{- 2 \ell (d+r)}$ as $t \to \infty$, we have
    $e^{\Psi_- (t)}  h(t) = t^{-r-1+ o(1)}$ as $t \to \infty$.
    Together with $e^{\Psi_- (t)}  h(t) \sim e^{\xi(0)} h(0) t^{d-1}$ as $t \downarrow 0$, we have $I_\infty = \int_0^\infty e^{\Psi_- (t)}  h(t) \dd t < \infty$ as $r>0$.
    By the same argument as in the case $\ell \in [0, \lstarstar)$, for any fixed $\tilde{\nu} \in (1, \nu)$ there exist $t_0>0$ and $c>0$ such that $e^{- \Psi_- (t)} \le c t^{- \tilde{\nu}}$ for $t \ge t_0$.
    Hence, for any $\rho > t_0$,
    $$
    f_- (\rho) \le f_- (t_0) + I_\infty \int_{t_0}^{\infty} c t^{- \tilde{\nu}} \dd t < \infty
    $$
    since $\tilde{\nu} > 1$.
    We thus conclude that $f_-$ is bounded, and so is $U_-$.

    \smallskip
    \noindent
    \underline{$\AA U_- \ge 1$.} 
    Since $f_-' (\rho) = \varphi (\rho)$ and $f_-'' (\rho) = h(\rho) - \psi_- (\rho) \varphi(\rho)$ for $\rho >0$, we have
    $$
    \AA U_- (x)
    = \pi(x)^{-2\ell} h(\abs{x}) + \pi(x)^{-2\ell} \varphi(\abs{x}) \bracket*{\psi (x) - \psi_- (\abs{x})}
    \ge \pi(x)^{-2\ell} h(\abs{x})
    \ge 1
    $$
    for $x \neq 0$.
    At the origin, $\AA U_-(0) = \pi(0)^{-2\ell}\Delta U_-(0) = \pi(0)^{-2\ell}h(0) = 1$.

    \smallskip
    By Lemma~\ref{lemma:explosion test} (2), we conclude that $\zeta < \infty$ a.s.\ when $\ell \in (\lstarstar, \infty)$. 
\end{proof}

\subsection{Non-explosion at \texorpdfstring{$\ell = \lstarstar$}{ℓ = ℓ**}}\label{subsec:explosion at lstarstar}

As discussed in Section~\ref{subsubsec:critical temperature thresholds} and Remark~\ref{remark:Khasminskii fails at lstarstar}, proving that $X$ is non-explosive at $\ell = \lstarstar$ requires a more delicate argument.
We bound the event that $X$ exits a large ball by time $T$ via a logarithmic cutoff function, which equals $1$ near the origin and vanishes outside the ball.
The main difficulty is then to control the largest increment of this function along the path of $X$ up to time $T$.
To this end, we construct a diffusion on a closed ball that behaves like $X$ in the interior and is reflected at the boundary.
As the extended generator of $X$ is in divergence form~\eqref{eq:generator on general functions}, this reflected diffusion admits the restriction of $\Pi$ to the ball, after normalisation, as its stationary measure and is reversible with respect to it.
By the Itô trick~\cite{LyonsZhang1994}, controlling the increment along the reflected diffusion then reduces to bounding the quadratic variations of two martingales, one for the process and one for its time reversal.

\begin{proof}[Proof of Theorem~\ref{thm:TLD explosion} \textup{($\ell = \lstarstar$)}]
    Fix $R_0 >0$, let $R > R_0$, and set $D = \B_{2R}$. 

    \smallskip
    \noindent
    \underline{Reflected diffusions.}
    Let $\mathfrak{n}$ be the inward unit normal vector to $\partial D$. 
    Consider the processes $(Y,L)$ with state spaces $\bar{D}$ and $\R_+$ satisfying the SDE 
    \begin{equation} \label{eq:reflected diffusion SDE}
        Y_t = x + \int_{0}^{t} b(Y_s) \dd s + \int_{0}^{t} \sigma (Y_s) \dd B_s + \int_{0}^{t} 2 a(Y_s) \mathfrak{n}(Y_s) \dd L_s 
        \myand 
        L_t = \int_0^t \1{Y_s \in \partial D} \dd L_s,
        \quad 
        t \ge 0,
    \end{equation}
    where $x \in D$, $a$ is defined in \eqref{eq:generator on general functions} and $L$ denotes the local time of $Y$ on $\partial D$. 
    Since $b$ and $\sigma$ are Lipschitz on $D$ and $a$ is bounded and uniformly elliptic on $D$, the SDE in~\eqref{eq:reflected diffusion SDE} has a unique strong solution for any $x \in D$ by \cite[Thm~4.4]{LionsSznitman1984}. 
    Thus, $Y$ is the diffusion in $D$ whose generator $\LL$ agrees with $\AA$ on $C_c^\infty(D)$, reflected at $\partial D$ in the conormal direction $2 a \mathfrak{n}$~\cite{PardouxWilliams1994}. 

    Let $\Pi_R$ be the probability measure defined by
    $$
    \Pi_R (\dd x) = \frac{\1{x \in D} \pi(x) \dd x}{Z_R}
    \myand
    Z_R = \int_D \pi(x) \dd x.
    $$
    As discussed in \cite[pp.~15--16]{PardouxWilliams1994}, since $\pi$ is positive on $\R^d$, the divergence form of the generator and the normal reflection at the boundary imply that $\Pi_R$ is the stationary measure of $Y$ and $Y$ is $\Pi_R$-reversible.
    Moreover, up to the first hitting time of~$\partial D$, $Y$ has the same law as $X$.

    \smallskip
    \noindent
    \underline{Itô trick.}
    Let $u \in C^\infty (\R^d)$ be a function that is constant in a neighbourhood of $\partial D$.
    Define
    $$
    M_t = u(Y_t) - u(Y_0) - \int_0^t \LL u(Y_s) \dd s
    \for
    t \ge 0.
    $$
    Then, Itô's formula shows that $(M_t)_{t \ge 0}$ is a square-integrable martingale under $\P_{\Pi_R}$ w.r.t.\ the natural filtration of $(Y_t)_{t \ge 0}$ and
    \begin{equation} \label{eq:reflected MG second moment}
        \E_{\Pi_R} \bracket*{M_t^2} = 2t \int \pi^{-2\ell} \abs{\nabla u}^2 \dd \Pi_R.
    \end{equation}

    Fix $T>0$.
    Define $\widehat Y_t = Y_{T-t}$ for $t \in [0,T]$.
    By reversibility, under $\P_{\Pi_R}$ the time-reversed process $\paren{\widehat Y_t}_{t \in [0,T]}$ has the same law as $\paren{Y_t}_{t \in [0,T]}$.
    Let $\widehat M$ be the corresponding martingale part of $u(\widehat{Y})$ w.r.t.\ the natural filtration of $(\widehat Y_t)_{t \in [0,T]}$.
    Applying the \textit{Itô trick} \cite{LyonsZhang1994} gives, for $0\le t\le T$, 
    \begin{align*}
        u(Y_t) - u(Y_0)
        & = 
        \frac12 \paren*{u(Y_t)-u( Y_0)-\int_0^t \LL u(Y_s)\,ds} - \frac12 \paren*{u(Y_0)-u(Y_t) - \int_{0}^t \LL u( Y_s)\,ds} 
        \\
        & = 
        \frac12 \paren*{u(Y_t)-u( Y_0)-\int_0^t \LL u(Y_s)\,ds} - \frac12 \paren*{u(\widehat Y_{T})-u(\widehat Y_{T-t})-\int_{T-t}^T\LL u(\widehat Y_s)\,ds} 
        \\
        & = 
        \frac12 M_t- \frac 12 \paren*{\widehat M_T-\widehat M_{T-t}}. 
    \end{align*}
    Then, applying Doob's martingale inequality to $M$ and $\widehat M$ and using \eqref{eq:reflected MG second moment}, we obtain
    \begin{equation} \label{eq:bound the largest increment of u}
        \E_{\Pi_R} \bigg[ \sup_{t \in [0,T]} \abs{u(Y_t)-u(Y_0)}^2 \bigg]
        \le C_1 T \int \pi^{-2\ell} \abs{\nabla u}^2 \dd \Pi_R
    \end{equation}
    for some constant $C_1 > 0$.

    \smallskip
    \noindent
    \underline{Logarithmic cutoff functions.}
    Let $\chi: \R \to [0,1]$ be a smooth function such that $\chi = 1$ on $(-\infty, 0]$ and $\chi = 0$ on $[1, \infty)$.
    Define
    $$
    u_R (x) = \chi \paren*{\frac{\log (\abs{x} / R_0)}{\log (R/R_0)}}
    \for
    x \neq 0,
    $$
    and extend it by $1$ near the origin.
    Then, $u_R= 1$ on $\B_{R_0}$, $u_R= 0$ on $\B_{R}^c$, and, for some constant $C_2 >0$,
    \begin{equation} \label{eq:bound gradient log cutoff}
        \abs{\nabla u_R (x)} \le \frac{C_2}{\abs{x} \log (R / R_0)} \1{R_0 < \abs{x} < R}.
    \end{equation}
    By \eqref{AP}, there is a constant $C_3>0$ such that $\pi(x)^{1-2\ell} \le C_3 \abs{x}^{2-d}$ on $\B_{R_0}^c$.
    Combining this with \eqref{eq:bound the largest increment of u}  and \eqref{eq:bound gradient log cutoff}, we obtain
    $$
    \E_{\Pi_R} \bracket*{\sup_{t \in [0,T]} \abs{u_R (Y_t)-u_R (Y_0)}^2}
    \le \frac{C_4 T}{Z_R \paren*{\log (R / R_0)}^2} \int_{R_0}^R \rho^{2 - d} \rho^{-2} \rho^{d-1} \dd \rho
    = \frac{C_4 T}{Z_R \log (R / R_0)}
    $$
    for some constant $C_4 > 0$ independent of $R$. 

    \smallskip
    \noindent
    \underline{Bounding exit probabilities.}
    Let $\zeta_R = \inf \set*{t>0: \abs{X_t} > R}$.
    Since $Y$ has the same law as $X$ before the first hitting time of~$\partial D$, and since $\sup_{t \in [0,T]} \abs{u_R(Y_t) - u_R(Y_0)} = 1$ whenever $Y$ starts inside $\B_{R_0}$ and leaves $\B_{R}$ by time $T$, we have
    $$
    \int_{\B_{R_0}} \P_x \paren*{\zeta_R \le T} \pi(x) \dd x
    \le Z_R \, \E_{\Pi_R} \bracket*{\sup_{t \in [0,T]} \abs{u_R(Y_t)-u_R(Y_0)}^2}
    \le \frac{C_4 T}{\log (R / R_0)}.
    $$
    Since $\zeta_R \uparrow \zeta$ as $R \to \infty$ $\P_x$-almost surely for every $x \in \R^d$, letting $R \to \infty$ and applying the dominated convergence theorem yields $\int_{\B_{R_0}} \P_x \paren*{\zeta \le T} \pi(x) \dd x = 0$. 
    Letting $R_0 \to \infty$ and $T \to \infty$, the monotone convergence theorem gives $\int_{\R^d} \P_x \paren*{\zeta < \infty} \pi(x) \dd x = 0$, which implies $\P_x (\zeta < \infty) =0$ for Lebesgue almost every $x \in \R^d$.

    It remains to extend this to every $x \in \R^d$.
    Since $\pi \in C^2(\R^d)$ has locally Hölder continuous second-order derivatives, $X$ admits a transition density before explosion by \cite[Thm~1.1]{Kent1978}.
    Therefore, for any $T>0$ and $\eps \in (0,T)$, by the Markov property at time $\eps$, we have
    $$
    \P_x (\eps < \zeta \le T)
    = \E_x \bracket*{\1{\eps < \zeta} \P_{X_\eps} (\zeta \le T - \eps)}
    = 0.
    $$
    As $\zeta >0$ $\P_x$-almost surely, letting $\eps \downarrow 0$ and applying the monotone convergence theorem gives that ${\P_x (\zeta \le T) = 0}$.
    Since $T>0$ is arbitrary, we conclude $\P_x (\zeta < \infty) = 0$, which completes the proof.
\end{proof}

\section{Moment thresholds} \label{sec:proof moments}

In this section, we establish the moments properties presented in Theorem~\ref{thm:TLD moments}. 
Many of the Lyapunov functions used in our proofs are built from the following function, which agrees with a power of the norm outside a neighbourhood of the origin. 
Let $\gamma \in\R\setminus\{0\}$ and $\VV_\gamma: \R^d \to (0,\infty)$ satisfy $\VV_\gamma\in C^2(\R^d)$ and $\VV_\gamma (x) = \abs{x}^{\gamma}$ for $x \in \bar{\B}_2^c$. 
If $\gamma > 0$, we further assume that $\VV_\gamma: \R^d \to [1,\infty)$. 
Recall that $\alpha = \ell (d+r) - 1$, we then have 
\begin{equation} \label{eq:generator on polynomial}
    \AA \VV_\gamma(x) = \theta_\ell(x)^2 \gamma \paren*{\gamma - (r - 2\alpha) + \eps_\ell(x)}  \abs{x}^{\gamma + 2 \alpha} \for x \in \bar{\B}_2^c,
\end{equation}
where $\theta_\ell(x) := \paren{\abs{x}^{d+r} \pi(x)}^{- \ell}$ and $\eps_\ell(x) := (1-2\ell) \paren*{\inner{\nabla \log \pi (x), x} + (d+r)}$ for $x \neq 0$.
Because $c_\pi = 1$, \eqref{AP} gives 
\begin{equation} \label{eq:generator poly error terms}
    \theta_\ell(x) \to 1 \myand \eps_\ell(x) \to 0 \as \abs{x} \to \infty.
\end{equation}

We present the moment thresholds for $\ell \in [0, \lstar]$ and $\ell \in (\lstar, \lstarstar]$ in Sections~\ref{subsec:moments below lstar} and~\ref{subsec:moments Bessel}, respectively. 

\subsection{Moments for \texorpdfstring{$\ell \in [0, \lstar]$}{ℓ ∈ [0, ℓ*]}} \label{subsec:moments below lstar}
We apply the same spatial transformation as in Section~\ref{subsubsec:heuristics moments}. 
For a general $\pi$ satisfying \eqref{AP}, rather than an exactly polynomially decaying one, the transformed process is no longer  dominated by a Brownian motion on $\R$ with negative drift but a time-changed one. 
This path domination lets us deduce the moment properties from the exponential moments of a Brownian motion.
Recall that $\alpha = \ell (d+r) - 1 \in [-1,0]$ when $\ell \in [0, \lstar]$. 

\begin{proof}[Proof of Theorem~\ref{thm:TLD moments} (1) and (2)]
    Consider a function $\Phi: \R^d \to (0,\infty)$ such that $\Phi \in C^2 (\R^d)$ and $\Phi(x) = f_\ell (\abs{x})$ for $x \in \bar{\B}_2^c$, where
    \begin{equation*}
        f_\ell(\rho)
        =
        \begin{cases}
            {\rho^{-\alpha}} / \paren*{\sqrt{2} \abs{\alpha}}, &\myif \ell \in [0, \lstar), \\
            \log \rho / \sqrt{2}, &\myif \ell = \lstar,
        \end{cases}
    \end{equation*}
    for $\rho \ge 2$.
    Note that $f_{\ell}'(\rho) \to f_{\lstar}'(\rho)$ and $f_{\ell}'' (\rho) \to f_{\lstar}'' (\rho)$ as $\ell \uparrow \lstar$ for $\rho \ge 2$. 
    Define
    $$
    M_t = \Phi(X_t) - \Phi(x) - \int_{0}^{t} \AA \Phi (X_s) \dd s
    \myand
    \inner{M}_t = \int_{0}^{t} 2 \pi(X_s)^{-2\ell} \abs{\nabla \Phi(X_s)}^2 \dd s.
    $$
    Itô's formula shows that $(M_t)_{t \ge 0}$ is a continuous $\P_x$-local martingale with quadratic variation $\paren*{\inner{M}_t}_{t \ge 0}$.
    A direct computation shows that, for $x \in \bar{\B}_2^c$,
    $$
    \AA \Phi (x)
    =
    - \frac{\theta_\ell(x)^2 \paren*{r - \alpha - \eps_\ell(x)}}{\sqrt{2}}  \abs{x}^{\alpha}
    \myand
    2\pi(x)^{-2\ell}\abs{\nabla \Phi(x)}^2 = \theta_\ell(x)^2.
    $$
    Together with \eqref{eq:generator poly error terms}, this shows that $K = \sup_{y \in \R^d} \abs{\AA \Phi(y)}$ and $K' = \sup_{y \in \R^d} 2 \pi(y)^{-2\ell} \abs{\nabla \Phi(y)}^2$ are both finite, and moreover $K' \ge 1$.
    By the Dambis-Dubins-Schwarz theorem (see e.g.\ \cite[Chap.~V, Thm~1.7]{RevuzYor1999}), on a possibly enlarged probability space, there exists a standard one-dimensional Brownian motion $W$ such that $M_t = W_{\inner{M}_t}$ for every $t \ge 0$. 
    Consequently,
    \begin{equation}\label{eq:bound by BM}
        \Phi(x) - K t + W_{\inner{M}_t}
        \le
        \Phi(X_t)
        \le 
        \Phi(x) + K t + W_{\inner{M}_t}
        \myforall
        t \ge 0
    \end{equation}
    holds $\P_x$-almost surely. 

    \smallskip
    \noindent
    \underline{Finite moments.}
    Since $\inner{M}_t \le K' t$, setting $S_t = \sup_{s \in [0,K' t]} W_s$ and applying \eqref{eq:bound by BM} gives
    $$
    \Phi(X_t) \le \Phi(x) + K t + S_t \quad \text{$\P_x$-almost surely}.
    $$
    For any $\lambda > 0$, we have $\E \bracket{\exp \paren{\lambda S_t^{\kappa}}} < \infty$ for every $\kappa \in [0,2)$ and $t \ge 0$, and also for $\kappa = 2$ provided $t < \frac{1}{2\lambda K'}$. 
    When $\ell \in [0, \lstar)$, we set $\tilde \kappa = p/\abs{\alpha}$ and $\tilde \lambda=(\sqrt2\abs{\alpha})^{\tilde \kappa}$. 
    Then, 
    \begin{align*}
        \E_x \bracket*{e^{\abs{X_t}^p }}
        \le e^{2^p} + \E_x \bracket*{e^{\tilde \lambda \Phi(X_t)^{\tilde \kappa}}} 
        \le e^{2^p} + e^{\tilde \lambda 2^{\tilde \kappa} (\Phi(x)+Kt)^{\tilde \kappa}} \E \bracket*{e^{\tilde \lambda 2^{\tilde \kappa} S_t^{\tilde \kappa}}},
    \end{align*}
    where the second inequality uses $(A+B)^{\tilde \kappa}\le2^{\tilde \kappa}(A^{\tilde \kappa}+B^{\tilde \kappa})$ for $A, B \ge 0$.
    Thus, for $t \ge 0$
    $$
    \E_x \bracket*{\exp \paren{\abs{X_t}^p }} 
    \begin{cases}
        < \infty, & \myif p \in [0,2 \abs{\alpha}), \\
        < \infty, & \myif p =2 \abs{\alpha} \myand t < {1}/ \paren*{16 \abs{\alpha}^2 K'}. 
    \end{cases}
    $$
    Similarly, when $\ell = \lstar$, we have for any $p \ge 0$, 
    $$
    \E_x \bracket*{ \abs{X_t}^p }
    \le
    2^p + \E_x \bracket*{ e^{\sqrt{2} p \Phi(X_t)} } 
    \le 2^p + e^{\sqrt{2} p (\Phi(x) + Kt)} \E \bracket*{ e^{\sqrt{2} p S_t} } 
    < \infty. 
    $$
    These justify the finite moments assertions in Theorem~\ref{thm:TLD moments} (1) and (2). 

    \smallskip
    \noindent
    \underline{Infinite moments.}
    Let $\rho_0 \ge 2$ be such that $f_\ell(\rho_0) \ge \sup_{\abs{y} \le 2} \Phi(y)$. 
    Fix $t >0$ and $\rho > \rho_0$.
    By \eqref{eq:generator poly error terms}, there is a constant $R > 2$ with $\frac{1}{2} \le \theta_\ell^2 \le \frac{3}{2}$ on $\bar{\B}_{R}^c$.
    Set $T_1 = t/2$, $T_2 = 3t/2$, and
    $$
    R_1 = R_1(t) = 
    \begin{cases}
        \paren*{ R^{\abs{\alpha}} + \sqrt{2} \abs{\alpha} (1+K t)}^{{1}/{\abs{\alpha}}}, & \myif \ell \in [0, \lstar), \\
        R e^{\sqrt{2} (K t + 1)}, & \myif \ell = \lstar,
    \end{cases}
    $$
    so that $f_\ell(R_1) = f_\ell(R) + Kt + 1$.
    Let $x \in \bar{\B}_{R_1}^c$ and $G= G_1 \cap G_2$, where
    $$
    G_1 = \set*{\inf_{s \le T_2} W_s > -1}
    \myand
    G_2 = \set*{\inf_{s \in [T_1, T_2]} W_s \ge f_\ell(\rho) + K t}.
    $$
    We claim that $G \subset \set{\abs{X_t} \ge \rho}$.
    Let $\tau := \inf \set{s: \abs{X_s} \le R}$, and suppose toward a contradiction that $\tau \le t$ on $G$.
    By \eqref{eq:bound by BM} and the choice of $R_1$,
    \begin{align*}
        \inf_{s \le T_2} W_s
        \le W_{\inner{M}_\tau}
        \le f_\ell(R) - f_\ell(R_1) + K \tau
        = K(\tau - t) - 1
        \le -1,
    \end{align*}
    which contradicts $G_1$.
    Hence $\tau > t$ on $G$, so that $\inner{M}_t \in [T_1, T_2]$ on $G$, which forces
    $$
    \Phi(X_t)
    \ge \Phi(x) - Kt + W_{\inner{M}_t}
    \ge - Kt + f_\ell(\rho) + K t
    = f_\ell(\rho).
    $$
    Since $f_\ell(\rho) > f_\ell(\rho_0) \ge \sup_{\abs{y} \le 2} \Phi(y)$, this forces $\abs{X_t} > 2$, so that $f_\ell (\abs{X_t}) = \Phi(X_t) \ge f_\ell(\rho)$ and hence $\abs{X_t} \ge \rho$ by the monotonicity of $f_\ell$, which proves the claim.
    With $a = f_\ell(\rho) + K t + 1$, observe that
    $$
    \set*{\inf_{s \le T_1} W_s > -1, W_{T_1} \ge a, \inf_{s \in [T_1, T_2]} W_s \ge W_{T_1} - 1}
    \subset G.
    $$
    Define $C_1 (t) = \frac{2}{\sqrt{\pi t}} e^{- (Kt+3)^2/t} \, \P \paren*{\inf_{s \le t} W_s \ge - 1}$, $C_2 (t) = 1/t$ and $C_3 (t) = 2(Kt+3)/t$. 
    Applying the Markov property together with the reflection principle gives 
    \begin{equation} \label{eq:uniform lower bound on tail}
        \begin{aligned}
            \inf_{x \in \bar{\B}_{R_1 (t)}^c} \P_x (\abs{X_t} \ge \rho)
            &\ge \P \big( \inf_{s \le T_1} W_s > -1, W_{T_1} \ge a \big) \P \big( \inf_{s \le t} W_s \ge - 1 \big) \\
            &= \P \paren*{ W_{T_1} \in \left[a,a+2\right)} \P \big( \inf_{s \le t} W_s \ge - 1 \big)\\
            &\ge C_1(t) \exp \paren*{-C_2(t) f_\ell (\rho)^2 - C_3(t) f_\ell (\rho)},
        \end{aligned}
    \end{equation}
    where the final inequality comes from $\P \paren*{ W_{T_1} \in [a, a+2)} \ge \frac{2}{\sqrt{2 \pi T_1}} \exp \paren*{-\frac{(a+2)^2}{2 T_1}}$. 

    For any $x \in \R^d$ and $t>0$, set $s = t/2$ and $\tilde{C}_1(s,x) = C_1(s) \P_x \paren*{X_{s} \in \bar{\B}_{R_1(s)}^c}>0$.
    Applying the Markov property at time $s$ and \eqref{eq:uniform lower bound on tail} gives
    \begin{align*}
            \P_x (\abs{X_t} \ge \rho)
        \ge \E_x \bracket*{ \1{X_{s} \in \bar{\B}_{R_1(s)}^c} \P_{X_{s}} \paren*{ \abs{X_{s}} \ge \rho }}
        \ge \tilde{C}_1(s,x) \exp \paren*{-C_2(s) f_\ell (\rho)^2 - C_3(s) f_\ell (\rho)}
    \end{align*}
    for $\rho > \rho_0$. 
    When $\ell \in [0, \lstar)$, we have for every $p > 2 \abs{\alpha}$ or $p = 2 \abs{\alpha}$ with $t > 1/\abs{\alpha}^2$,
    \begin{align*}
    \E_x \bracket*{e^{ \abs{X_t}^p}}
    = 
    1 + \int_0^\infty e^{\rho} \, \P_x (\abs{X_t}^p \ge \rho) \dd \rho 
    \ge 
    \tilde C_1 \paren*{s,x} \int_{\rho_0^p}^\infty \exp \paren*{\rho - \frac{1}{\abs{\alpha}^2 t} \rho^{\frac{2\abs{\alpha}}{p}} - \frac{C_3(s)}{\sqrt{2} \abs{\alpha}} \rho^{\frac{\abs{\alpha}}{p}}} \dd \rho
    = \infty. 
    \end{align*}
    Similarly, when $\ell = \lstar$, we have for any $p>0$,
    $$
    \E_x \bracket*{e^{\abs{X_t}^p}}
    \ge 
    \tilde C_1 \paren*{s,x} \int_{\rho_0^p}^\infty \exp \paren*{\rho - \frac{C_2(s)}{2 p^2} (\log \rho)^2 - \frac{C_3(s)}{\sqrt{2} p} \log \rho} \dd \rho
    = \infty. 
    $$
    The proof of Theorem~\ref{thm:TLD moments} (1) and (2) is now complete. 
\end{proof}

\subsection{Moments for \texorpdfstring{$\ell \in (\lstar, \lstarstar]$}{ℓ ∈ (ℓ*, ℓ**]}} \label{subsec:moments Bessel}

As in Section~\ref{subsubsec:heuristics moments}, we work with a function $\Phi: \R^d \to (0,\infty)$ satisfying $\Phi \in C^2(\R^d)$ and $\Phi(x) = \frac{1}{\sqrt{2} \alpha} \abs{x}^{-\alpha}$ for $x \in \bar{\B}_2^c$. 
We claim that for general $\pi$ satisfying~\eqref{AP}, $\Phi(X)$ behaves approximately like a Bessel process of dimension $\deltastar = {r}/{\alpha} \in [2, \infty)$. 
To justify this, we compute the action of the extended generator of $X$ on $\varphi \circ \Phi$ for any $\varphi \in C^2(\R_+)$, obtaining
\begin{equation} \label{eq:generator almost Bessel}
    \AA \paren*{\varphi \circ \Phi} (x)
    = \theta_\ell(x)^2 \bracket*{\frac{\delta_\ell (x) - 1}{2 z} \varphi'(z) + \frac{1}{2} \varphi''(z)}
    \for
    \abs{x} > 2,
\end{equation}
where $\theta_\ell(x) = \paren{\abs{x}^{d+r} \pi(x)}^{- \ell}$, $\delta_\ell (x) := \paren*{r - \eps_\ell (x)}/ \alpha$ and $z = \Phi(x)$.
Since \eqref{AP} gives
\begin{equation} \label{eq:generator almost Bessel error}
    \theta_\ell(x) \to 1
    \myand
    \delta_\ell(x) \to \deltastar
    \as \abs{x} \to \infty,
\end{equation}
this generator identity implies that $\Phi(X)$ evolves like a Bessel process of dimension $\deltastar$ when $X$ is far from the origin. 

Combined with the identity, for $p \in [0,r)$,
\begin{equation} \label{eq:equivalence Bessel neg moments}
    \E_x \bracket*{\abs{X_t}^p \1{\abs{X_t} > 2}}
    = \paren{\sqrt{2} \alpha}^{- \nu} \E_x \bracket*{\Phi(X_t)^{-\nu} \1{\abs{X_t} > 2}},
    \quad
    \nu = p/\alpha \in [0, \deltastar),
\end{equation}
we attempt to deduce $\E_x \bracket*{\abs{X_t}^p} < \infty \iff p \in [0,r)$ from the fact that a Bessel process has finite negative moments only of degree less than its dimension.
However, unlike in the case $\ell \in [0, \lstar]$, here it is not enough to view $\Phi(X)$ as a time-changed Bessel process to directly deduce its moment properties. 
This is because a Bessel process's infimum does not have the same negative moments as the process itself.
Let $Z$ be a Bessel process with dimension $\delta>0$ and $Z_0 = z \ge 0$. 
By \cite[Chap.~XI, Cor.~1.4]{RevuzYor1999} and \cite[Part II, Chap.~4, (1.2.2)]{BorodinSalminen2002}, we have: for $\delta > 2$ and $z>0$,
\begin{equation} \label{eq:Bessel lose moments}
    \E \bracket{Z_t^{-\eta}} < \infty \iff \eta \in [0,\delta),
    \quad \text{but} \quad
    \E \bracket{\paren{\inf_{s \in [0,t]} Z_s}^{-\eta}} < \infty \iff \eta \in [0,\delta-2). 
\end{equation}
By \eqref{eq:equivalence Bessel neg moments}, this means that the path-domination argument used for $\ell \in [0, \lstar]$ only yields finite moments for $p \in [0, r - 2 \alpha)$ when $\ell \in (\lstar, \lstarstar)$.

Alternatively, one can apply a version of Grönwall's inequality for expectations (see e.g.\ \cite[Lemma~2.10]{BM2026}), for which we seek a bound of the form $\AA \VV_p \le c \VV_p$ on $\R^d$ for some $c \ge 0$.
By \eqref{eq:generator on polynomial},
$$
\AA \VV_p(x) = p \paren*{p - (r - 2\alpha) + \eps_\ell(x)} \theta_\ell(x)^2 \abs{x}^{p + 2 \alpha} \for x \in \bar{\B}_2^c.
$$
For $\ell \in (\lstar, \lstarstar]$, we have 
\(\alpha=\ell(d+r)-1>0\)
 and $p - (r - 2\alpha) + \eps_\ell(x) < 0$ for $x$ large enough only when $p \in (0, r - 2\alpha)$.
Consequently, a bound of the form $\AA \VV_p \le c \VV_p$ fails to hold for $p \in (r - 2\alpha, r)$, and a different Lyapunov function is required.

To resolve this issue, we build the Lyapunov function directly out of the closed-form expression of the negative moments of Bessel processes, as we motivated in Section~\ref{subsec:heuristics Lyapunov function}. 
Let $M$ denote \textit{Kummer's confluent hypergeometric function} (see e.g.\ \cite[Sec.~13]{AbramowitzStegun1972}), given by
$$
M(a,b,z) = \sum_{n=0}^{\infty} \frac{(a)_n z^n}{(b)_n n!}
\for 0 \le a < b
\myand 
z \in \R,
$$
with $(a)_0 = 1$ and $(a)_n = \prod_{k=0}^{n-1}(a+k)$ for $n \ge 1$.
Since $Z_s^2/s$ follows, for each $s>0$, the non-central chi-squared distribution with $\delta$ degrees of freedom and non-centrality parameter $\lambda = z^2/s$ \cite[Chap.~XI, Cor.~1.4]{RevuzYor1999}, computing its negative moments and applying Kummer's transformation \cite[(13.1.27)]{AbramowitzStegun1972} gives, for any $\eta \in (0,\delta)$ and $(s,z) \in (0,\infty) \times [0,\infty)$,
\begin{equation} \label{eq:closed forms of Bessel negative moments}
    U_{\delta, \eta}(s,z) 
    := \E \bracket*{Z_s^{-\eta}}
    = s^{-\eta/2} \sum_{n=0}^{\infty} e^{-\lambda/2} \frac{(\lambda/2)^n}{n!} \frac{\Gamma \paren*{\frac{\delta-\eta}{2} + n}}{2^{\eta/2} \Gamma \paren*{\frac{\delta}{2} + n}}
    = K_{\delta, \eta} s^{-{\eta}/{2}} M \paren*{{\eta}/{2}, {\delta}/{2}, -{z^2}/ \paren{2s}},   
\end{equation}
where $K_{\delta, \eta} = \frac{\Gamma \paren*{\paren{\delta - \eta}/{2}}}{2^{\eta / 2} \Gamma \paren*{\delta/2}}$. 
The following lemma collects the properties of $U_{\delta, \eta}$ that we will need and we defer its proof to Appendix~\ref{appendix:proof negative moments of Bessel processes}.

\begin{lemma} \label{lemma:properties of Bessel neg moments}
    For any $0 < \eta < \delta$, 
    \begin{enumerate}
        \itemsep 5pt
        \item there exist constants $0 < c_{\delta, \eta} \le C_{\delta, \eta} < \infty$ such that 
        \begin{equation} \label{eq:Bessel neg moments two sided bounds}
            c_{\delta, \eta} \paren*{s+z^2}^{-\eta/2}
            \le 
            U_{\delta, \eta} (s,z)
            \le
            C_{\delta, \eta} \paren*{s+z^2}^{-\eta/2}
            \for 
            s >0, \ z \ge 0;
        \end{equation}
        \item we have, on $(0, \infty)^2$,
        \begin{equation} \label{eq:Bessel neg moments derivatives}
            \partial_z U_{\delta, \eta} = - \eta z U_{\delta+2, \eta+2}, \quad 
            \partial_z^2 U_{\delta, \eta} = - \eta U_{\delta+2, \eta+2} + \eta (\eta + 2) z^2 U_{\delta+4, \eta+4} 
        \end{equation}
        and 
        \begin{equation} \label{eq:Bessel neg moments PDE}
            \partial_s U_{\delta, \eta} = \frac{\delta - 1}{2 z} \partial_z U_{\delta, \eta} + \frac{1}{2} \partial_z^2 U_{\delta, \eta};
        \end{equation}
        \item there exists a constant $C(\delta, \eta)>0$ such that $\abs*{\partial_s U_{\delta, \eta}} \le C(\delta, \eta) z^{-1} \abs*{\partial_z U_{\delta, \eta}}$ for all $s,z>0$. 
    \end{enumerate}
\end{lemma}

Fix $p \in (0,r)$, $\eps_0 >0$ and $T>0$.
Recall that when $\ell \in (\lstar, \lstarstar]$, we have $\nu = p/\alpha \in (0,\deltastar)$ and $\deltastar = r/\alpha \ge 2$.
Choose $\delta' \in (\nu, \deltastar)$ and $K>0$, and set $s(t) = T-t+\eps_0$.
Define $\UU \in C^{1,2} ([0,T] \times \R^d)$ by
\begin{equation} \label{eq:Bessel Lyapunov def}
    \UU (t,x) = U_{\delta', \nu} \paren*{s(t), \Phi(x)} + K(T-t).
\end{equation}
We thus have $\UU (t, \cdot)$ is comparable to $\abs{x}^p$ by \eqref{eq:Bessel neg moments two sided bounds} and is amenable to a Grönwall-type argument.
\begin{proposition} \label{prop:Bessel Lyapunov is non-positive}
    There exists $K>0$ independent of $\eps_0$ and $T$ such that $$
    \paren*{\partial_t + \AA} \, \UU \le 0 \myon [0,T] \times \R^d.$$ 
\end{proposition}

\begin{proof}
    For $(t,x) \in [0,T] \times \R^d$, a direct computation gives
    \begin{equation} \label{eq:bound Bessel Lyapunov}
        \paren*{\partial_t + \AA} \, \UU (t,x)
        = -K - \partial_s U_{\delta', \nu} (s(t), \Phi(x))
        + \AA \paren*{U_{\delta', \nu} (s(t), \cdot) \circ \Phi} (x).
    \end{equation}
    Write $z = \Phi (x)$, and abbreviate $s = s(t)$ and $U = U_{\delta', \nu} \paren*{s,z}$.
    Combining \eqref{eq:Bessel neg moments PDE} with \eqref{eq:generator almost Bessel} yields
    \begin{equation} \label{eq:tail bound Bessel Lyapunov}
        \begin{aligned}
            \paren*{\partial_t + \AA} \, \UU
            &\le - \paren*{ \frac{\delta' -1}{2z} \partial_z U + \frac{1}{2} \partial_z^2 U}
                + \theta_\ell(x)^2 \paren*{\frac{\delta_\ell(x) -1}{2z} \partial_z U + \frac{1}{2} \partial_z^2 U} \\
            &=  \paren*{\theta_\ell(x)^2 - 1} \paren*{ \frac{\delta' -1}{2z} \partial_z U + \frac{1}{2} \partial_z^2 U}
                + \theta_\ell(x)^2 \paren*{\frac{\delta_\ell(x) -1}{2z} \partial_z U - \frac{\delta' -1}{2z} \partial_z U} \\
            &\le \abs*{\theta_\ell(x)^2 - 1} \abs*{\partial_s U} - \theta_\ell(x)^2 \frac{\delta_\ell(x) - \delta'}{2z} \abs*{\partial_z U},
        \end{aligned}
    \end{equation}
    where the last step uses that $\partial_z U < 0$, as shown in \eqref{eq:Bessel neg moments derivatives}.
    Set $\lambda = \deltastar - \delta' > 0$.
    Then \eqref{eq:generator almost Bessel error} yields the existence of $R>2$ such that
    $$
    \delta_\ell (x) - \delta' > \frac{\lambda}{2}
    \myand
    \abs*{\theta_\ell(x)^2 - 1} < \frac{\lambda}{8 C(\delta', \nu)} \wedge \frac{1}{2}
    \for
    x \in \bar{\B}_R^c.
    $$
    Plugging these bounds into \eqref{eq:tail bound Bessel Lyapunov} and using Lemma~\ref{lemma:properties of Bessel neg moments} (3), we obtain on $\bar{\B}_R^c$
    $$
    \paren*{\partial_t + \AA} \, \UU
    \le \frac{\lambda}{8} \frac{\abs{\partial_s U}}{C(\delta', \nu)} - \frac{\lambda}{8}  \frac{\abs*{\partial_z U}}{z}
    \le 0.
    $$

    It remains to study $\paren*{\partial_t + \AA} \, \UU$ on the compact subset $[0,T] \times \bar{\B}_R$. 
    By \eqref{eq:bound Bessel Lyapunov}, \eqref{eq:generator on general functions} and the chain rule,
    $$
    \paren*{\partial_t + \AA} \, \UU
        \le -K + \abs*{\partial_s U_{\delta', \nu}}
        + \pi^{-2\ell} \abs*{ \paren*{(1-2\ell) \inner{\nabla \log \pi, \nabla \Phi} + \Delta \Phi} \partial_z U + \abs*{\nabla \Phi}^2 \partial_z^2 U}.
    $$
    Define $I_R = \inf_{x \in \bar{\B}_R} \Phi(x) >0$ which is positive since $\Phi$ is continuous and positive on $\R^d$ and $\bar{\B}_R$ is compact.
    Lemma~\ref{lemma:properties of Bessel neg moments} then gives a constant $C>0$, independent of $\eps_0$ and $T$, such that for $x \in \bar{\B}_R$
    $$
    \abs{\partial_z U(s,z)} \le C I_R^{-(\nu+1)}
    \myand
    \max \set*{\abs{\partial_z^2 U(s,z)}, \abs{\partial_s U(s,z)}} \le C I_R^{-(\nu+2)}.
    $$
    Moreover, since $\pi, \Phi \in C^2(\R^d)$ and $\pi$ is positive, the quantities $\pi^{-2\ell}$, $\abs{\nabla \log \pi}$, $\abs{\nabla \Phi}$ and $\abs{\Delta \Phi}$ are uniformly bounded on $\bar{\B}_R$.
    Combining these bounds shows that 
    $$
    \sup_{(t,x) \in [0,T] \times \bar{\B}_R}  \paren*{\abs*{\partial_s U_{\delta', \nu}} + \abs{\AA \UU}} \le C(R)
    $$ 
    for some constant $C(R)>0$ independent of $\eps_0$ and $T$, and taking $K = C(R)$ completes the proof.
\end{proof}

The next proposition confirms that $\Pi$ is the unique invariant probability measure of $X$ and shows that the transition density of $X$ with respect to $\Pi$ is symmetric in the spatial variables. 
This symmetry will be used to prove the assertion about infinite moments in Theorem~\ref{thm:TLD moments} (3). 

\begin{proposition} \label{prop:TLD density}
    Suppose $0 \le \ell \le \lstarstar$. 
    There exists a continuous function $p: (0,\infty) \times \R^d \times \R^d \to (0,\infty)$ such that $\P_x (X_t \in \dd y) = p(t,x,y) \dd y$. 
    The function $q: (0,\infty) \times \R^d \times \R^d \to (0,\infty)$ defined by
    $$
    q(t, x, y) := \frac{p(t, x, y)}{Z_\pi^{-1} \pi(y)}, \mywhere Z_\pi := \int_{\R^d} \pi(x) \dd x,
    $$
    is symmetric in $x$ and $y$. 
    Moreover, 
    \begin{equation} \label{eq:convergence to inv measure}
        \lim_{t \to \infty} \int_A p(t,x,y) \dd y = Z_\pi^{-1} \int_A \pi(dy)
        \myforall
        A \in \BB (\R^d), 
    \end{equation}
    and $\Pi$ is the unique invariant probability measure of $X$. 
\end{proposition}

\begin{proof} 
    Since $X$ is non-explosive and \eqref{AP} holds, the existence and properties of $p$ follow from \cite[Thm~1.1]{Kent1978}.
    By \cite[Thm~4.1 \& Thm~6.2]{Kent1978}, $q(t, \cdot, \cdot)$ is symmetric in its spatial variables and \eqref{eq:convergence to inv measure} holds. 
\end{proof}

We are now ready to finish the proof of Theorem~\ref{thm:TLD moments}. 
\begin{proof}[Proof of Theorem~\ref{thm:TLD moments} (3)]
    Suppose $\ell \in (\lstar, \lstarstar]$. 
    Recalling that $\zeta_n := \inf \set{t > 0: \abs{X_t} \ge n}$ and $\zeta = \lim_{n \to \infty} \zeta_n$, this assumption guarantees that $\zeta = \infty$ almost surely.

    \smallskip
    \noindent
    \underline{$p \in [0,r)$.}
    Let $x \in \R^d$ and $T>0$.
    Applying Itô's formula localised at $\zeta_n$ and Proposition~\ref{prop:Bessel Lyapunov is non-positive} yields
    $$
    \E_x \bracket*{\UU \paren*{T \wedge \zeta_n, X_{T\wedge \zeta_n}}}
    =
    \UU \paren*{0, x} + \E_x \bracket*{\int_0^{T \wedge \zeta_n} \paren*{\partial_t + \AA} \UU (v, X_v) \dd v}
    \le
    \UU \paren*{0, x}.
    $$
    Combining Fatou's lemma with \eqref{eq:Bessel neg moments two sided bounds} then gives
    $$
    \E_x \bracket*{U_{\delta', \nu} (s(T), \Phi (X_T))}
    \le KT + C_{\delta', \nu} (T + \eps_0)^{-\nu/2},
    $$
    while \eqref{eq:Bessel neg moments two sided bounds} also gives
    $
    \E_x \bracket*{U_{\delta', \nu} (s(T), \Phi (X_T))}
    \ge
    c_{\delta', \nu} \E_x \bracket{\paren*{\eps_0 + \Phi(X_T)^2}^{-\nu/2}}.
    $
    Sending $\eps_0 \downarrow 0$ and applying the monotone convergence theorem gives 
    $$
    \E_x \bracket*{\Phi(X_T)^{-\nu}}
    \le
    c_{\delta', \nu}^{-1} \big( KT + C_{\delta', \nu} T^{-\nu/2} \big).
    $$
    Since this bound is uniform over $x\in\R^d$, \eqref{eq:equivalence Bessel neg moments} yields
    $$
    \sup_{x \in \R^d} \E_x \bracket*{\abs{X_T}^p}
    \le
    2^p + \paren*{\sqrt 2 \alpha}^{-\nu}  c_{\delta', \nu}^{-1} \paren*{KT + C_{\delta', \nu} T^{-\nu/2}}
    =: C(T).
    $$
    Fix $t_0>0$. 
    For any $t \ge t_0$, the Markov property applied at time $t - t_0$ gives
    $$
    \E_x \bracket*{\abs{X_t}^p}
    = \E_x \bracket*{\E_{X_{t - t_0}} \bracket*{\abs{X_{t_0}}^p} }
    \le C(t_0),
    $$
    which completes the proof of \eqref{eq:moments uniform in time and space above lstar}.

    \smallskip
    \noindent
    \underline{$p \in [r,\infty)$.}
    Let $t>0$ and $x \in \R^d$ be fixed.
    As in Proposition~\ref{prop:TLD density}, denote by $q:(0,\infty) \times \R^d \times \R^d \to (0,\infty)$ the transition density of $X$ with respect to the invariant measure $\Pi$.
    We claim that
    $$
    I(t,x) = \inf_{y\in \R^d} q(t, x, y) >0.
    $$
    For $R>0$, set $u = t/2$ and $L(R) = \inf_{w \in \bar{\B}_{R}} q \paren*{u, x, w}$.
    By Proposition~\ref{prop:TLD density}, $q$ is continuous and positive, and since $\bar{\B}_{R}$ is compact, it follows that $L(R)>0$. 
    Let $y \in \R^d$. 
    Using the Chapman–Kolmogorov equation together with the symmetry of $q$ in its spatial variables (see Proposition~\ref{prop:TLD density}), we obtain
    \begin{align*}
        q(t, x, y)
        \ge
        \int_{\bar{\B}_{R}} q \paren*{u, x, w} q \paren*{u, w, y} \Pi (\dd w)
        \ge
        L(R) \int_{\bar{\B}_{R}} q \paren*{u, y, w} \Pi (\dd w)
        =
        L(R) \P_y (X_{u}  \in \bar{\B}_{R}).
    \end{align*}
    Choose $\gamma \in (0,r)$ and set $R = \paren*{2 \sup_{w \in \R^d} \E_w \bracket*{\abs{X_u}^\gamma}}^{1/\gamma}$, which is finite by \eqref{eq:moments uniform in time and space above lstar}.
    The Markov inequality then gives
    $$
    \P_y (X_{u}  \in \bar{\B}_{R})
    =
    1 - \P_y \paren*{\abs{X_{u}} > R}
    \ge
    1 - \frac{\sup_{w \in \R^d} \E_w \bracket*{\abs{X_u}^\gamma}}{R^\gamma}
    = \frac{1}{2}.
    $$
    Hence $I(t,x) \ge L(R)/2 >0$, which proves the claim. 
    Thus, by \eqref{AP},
    $$
    \E_x \bracket*{\abs{X_t}^p}
    =
    \int_{\R^d} \abs{y}^p q(t, x, y) \Pi (\dd y)
    \ge I(t,x) \int_{\R^d} \abs{y}^p \Pi (\dd y)
    = \infty
    \myforall
    p \in [r, \infty),
    $$ 
    which completes the proof.
\end{proof}

\section{Ergodicity} \label{sec:proof ergodicity}

In this section, we prove the ergodicity results.
We first derive some general bounds on Wasserstein distances in Section~\ref{subsec:bounding wasserstein} and then prove Theorem~\ref{thm:TLD convergence rates} in Section~\ref{subsec:proof ergodicity}.

\subsection{Bounding the Wasserstein distance} \label{subsec:bounding wasserstein}
We establish general upper and lower bounds on Wasserstein distances in Proposition~\ref{prop:upper and lower bounds wasserstein distance}. 
From the upper bound we deduce a control of the Wasserstein distance by the $f$-variation distance in \eqref{eq:f-var control wasserstein}, while the lower bound yields a lower bound on the convergence in Wasserstein distance in Lemma~\ref{lemma:lower bound Wasserstein convergence}, the counterpart of the lower bound on $f$-variation in \cite[Lemma~5.1]{BM2026}.
Before introducing these results, let us first define $f$-variation and Wasserstein distances on general metric spaces. 

Let $(\XX, d)$ be a Polish space endowed with the Borel $\sigma$-algebra. 
For any measurable function $f: \XX \to [1, \infty)$, the \textit{$f$-variation} of a signed measure $\mu$ on $(\XX, \BB (\XX))$ is 
$$
\norm{\mu}_f = \sup \set{ \abs{\int g \, \dd \mu}: g: \XX \to \R \text{ measurable and } \abs{g} \le f}.
$$ 
When $f \equiv 1$, it is the \textit{total variation} denoted by $\norm{\cdot}_{\TV}$. 
For any two probability measures $\mu, \nu$ on $(\XX, \BB (\XX))$, we have $\norm{\mu - \nu}_{\TV} = 2\sup_{A \in \BB(\XX)} \abs{\mu(A) - \nu(A)}$. 

For $p \in (0, \infty)$, let $\PP_p(\XX)$ denote the space of probability measures on $\XX$ with finite $p$-th moment, that is,  $\int_\XX d(x_0, x)^p \mu(\dd x) < \infty$ for some (equivalently, every) $x_0 \in \XX$. 
The \textit{$L^p$-Wasserstein distance} between probability measures $\mu, \nu \in \PP_p(\XX)$ is 
$$
\WW_p(\mu, \nu) = \inf_{\eta \in C(\mu, \nu)} \paren*{\int_{\XX \times \XX} d(x,y)^p \, \eta (\dd x, \dd y) }^{\frac{1}{p \vee 1}},
$$
where $C(\mu, \nu)$ is the collection of all couplings of $\mu$ and $\nu$. 

\begin{proposition} \label{prop:upper and lower bounds wasserstein distance}
    Let $\mu, \nu\in \PP_p(\XX)$ be probability measures on $\XX$ for $p \in (0, \infty)$.
    Then, for any $x_0 \in \XX$,
    \begin{equation} \label{eq:wasserstein upper bound}
        \WW_p (\mu, \nu)^{p \vee 1}
        \le
        2^{p \vee 1 - 1} \int_{\XX} d(x, x_0)^p \, \dd \abs{\mu - \nu} (x).
    \end{equation}
    Let $\varphi: \XX \to \R_+$ be Lipschitz continuous with constant $\norm{\varphi}_{\Lip} > 0$.
    For any $p \in (0,\infty)$, we have 
    \begin{equation} \label{eq:wasserstein lower bound}
        \abs*{\paren*{\int_\XX \varphi^{p} \dd \mu}^{\frac{1}{p \vee 1}} -  \paren*{\int_\XX \varphi^p \dd \nu}^{\frac{1}{p \vee 1}}} 
        \le
        \norm{\varphi}_{\Lip}^{p \wedge 1} \WW_p (\mu, \nu).
    \end{equation}
\end{proposition}

\begin{remark}
    For $p \in [1,\infty)$ and $\XX$ a locally compact Polish length space, \eqref{eq:wasserstein lower bound} was established in \cite[Prop.~7.29]{Villani2009} via displacement interpolation, whereas we prove \eqref{eq:wasserstein lower bound} directly via Minkowski's inequality without further assumption on $\XX$.
\end{remark}

\begin{proof}[Proof of Proposition~\ref{prop:upper and lower bounds wasserstein distance}]
    Since $x \le x^p$ on $[0,1]$ for any $p\in(0,1)$, the following inequality holds:
    \begin{equation} \label{fact:subadditivity}
        (A+B)^p \le A^p + B^p \quad \text{for all} \quad A,B \ge 0 \myand p \in (0,1).
    \end{equation}
    The assertion \eqref{eq:wasserstein upper bound} for $p \in [1, \infty)$ is proved in \cite[Thm~6.15]{Villani2009}, and the same argument, combined with \eqref{fact:subadditivity}, extends it to $p \in (0,1)$.

    Let $\eta \in C(\mu, \nu)$ be a coupling of $\mu$ and $\nu$. 
    Define the functions $\tilde \varphi: (x,y) \mapsto \varphi(x)$ and $\hat \varphi: (x,y) \mapsto \varphi(y)$, so that 
    $\int_\XX \varphi^p \dd \mu = \int_{\XX \times \XX} \tilde \varphi^p \dd \eta$ and $\int_\XX \varphi^p \dd \nu = \int_{\XX \times \XX} \hat \varphi^p \dd \eta$.
    Then, applying \eqref{fact:subadditivity} if $p \in (0,1)$, or Minkowski's inequality in $L^p(\eta)$ if $p \in [1,\infty)$ yields
    \begin{align*}
        \abs*{\paren*{\int_\XX \varphi^p \dd \mu}^{\frac{1}{p \vee 1}} -  \paren*{\int_\XX \varphi^p \dd \nu}^{\frac{1}{p \vee 1}}}
        & =
        \abs*{\paren*{\int_{\XX \times \XX} \tilde \varphi^p \dd \eta}^{\frac{1}{p \vee 1}} -  \paren*{\int_{\XX \times \XX} \hat \varphi^p \dd \eta}^{\frac{1}{p \vee 1}}} \\
        & \le
        \paren*{\int_{\XX \times \XX} \abs*{\varphi(x) - \varphi(y)}^p \eta \paren{\dd x, \dd y}}^{\frac{1}{p \vee 1}} \\
        & \le
        \norm{\varphi}_{\Lip}^{p \wedge 1} \paren*{\int_{\XX \times \XX} d(x,y)^p \eta \paren{\dd x, \dd y}}^{\frac{1}{p \vee 1}},
    \end{align*}
    where the last inequality comes from the Lipschitz continuity of $\varphi$. 
    Taking the infimum over $\eta \in C(\mu, \nu)$ yields \eqref{eq:wasserstein lower bound} and completes the proof. 
\end{proof}

To turn \eqref{eq:wasserstein upper bound} into an upper bound on the Wasserstein distance, we use the following integral representation of the $f$-variation, whose proof is deferred to Appendix~\ref{appendix:proof f-variation representation}.

\begin{lemma} \label{lemma:bound by f-var}
    Let $\eta$ be a signed measure on a measurable space $(E, \EE)$. 
    For any measurable function $f: E \to [1, \infty)$ such that $\int f \dd \abs{\eta} < \infty$, we have $\int f \dd \abs{\eta} = \norm{\eta}_f$. 
\end{lemma}
\noindent
Consequently, let $\mu, \nu \in \PP_p (\XX)$, and suppose there exist a measurable function $f: \XX \to [1, \infty)$, a constant $c > 0$ and a point $x_0 \in \XX$ such that $\int f \dd \abs{\mu - \nu} < \infty$ and $d(\cdot, x_0)^p \le c f$ on $\XX$.
Then Lemma~\ref{lemma:bound by f-var} yields
\begin{equation} \label{eq:f-var control wasserstein}
    \WW_p(\mu, \nu)^{p \vee 1} \le 2^{p \vee 1 - 1} c \norm{\mu - \nu}_f.
\end{equation}

The following lemma extends \cite[Lemma~4.8]{BMR2025}, established there for Markov chains and $p \in [1,\infty)$, to general Markov processes and every $p \in (0,\infty)$.

\begin{lemma} \label{lemma:lower bound Wasserstein convergence}
    Let $\xi = (\xi_t)_{t \ge 0}$ be a Markov process with an invariant measure $\Pi$ on $(\XX,d)$. 
    Assume that for some $p \in (0,\infty)$, we have $\Pi, \P_x (\xi_t \in \cdot)\in\PP_p(\XX)$ for all $x\in\XX$ and $t \ge 0$. 
    Let $W:\XX \to [1,\infty)$ be a Lipschitz function, that is, $\abs{W(x)-W(y)} \leq \norm{W}_{\Lip} d(x,y)$ for all $x,y \in \XX$ with the Lipschitz constant $\norm{W}_{\Lip} > 0$. 
    Let $g: [1,\infty) \to [1,\infty)$ be continuous and strictly increasing with $g(1) = 1$ and $\lim_{\rho \to \infty} g(\rho) = \infty$. Define $H = W^p (g \circ (2W))$ and assume the following conditions hold. 
    \begin{enumerate}
        \item There exists a continuous function $a_g: [1,\infty) \to (0,1]$, such that $A_g (\rho) = \rho a_g (\rho)$ is strictly increasing, $\lim_{\rho \uparrow \infty} A_g (\rho) = \infty$ and 
        $$
        \int_{\set*{g \circ W \ge \rho}} \paren*{W(x)/2}^p \, \Pi(\dd x) \ge a_g(\rho) 
        \myforall 
        \rho \in [1,\infty).
        $$
        \item There exists a function $v: \XX \times [0, \infty) \to [1,\infty)$, increasing in the second argument and satisfying 
        $\E_x \bracket*{ H(\xi_t)} \le v(x,t)$ for all $x \in \XX$ and $t \ge 0$. 
    \end{enumerate}
    Then the following bound holds for every $t > 0$ and $x \in \XX$:
    \begin{equation} \label{eq:Wasserstein convergence rate lower bound}
        \WW_p(\P_x (\xi_t \in \cdot), \Pi)^{p \vee 1}
        \ge 
        \frac{\paren*{a_g \circ A_g^{-1}\circ (2^{p \vee 1} v)} (x,t)}{2^{p \vee 1} \norm{W}_{\Lip}^p}.
    \end{equation}
\end{lemma}

\begin{proof}
    The argument follows that of \cite[Lemma~4.8]{BMR2025}, which we adapt to the present setting and extend to $p \in (0,1)$.
    Fix $x \in \XX$ and $t > 0$, and write $\mu_t := \P_x (\xi_t \in \cdot)$.
    For $s \in (1,\infty)$, define $f_{s}:\XX\to \R_+$ by 
    $$
    f_{s}(y) \coloneqq \paren*{W(y)-g^{-1}(s)/2}\1{W(y)\geq g^{-1}(s)/2}\for y\in\XX.
    $$ 
    Since $f_s(y)\geq W(y)/2$ when $W(y)\geq g^{-1}(s)$, Condition~(1) in Lemma~\ref{lemma:lower bound Wasserstein convergence} yields
    \begin{equation}
    \label{eq:Wasserstein_bound}
    \int_{\XX}f_{s}^p \dd \Pi \geq \int_{\{g\circ W\geq s\}} \paren*{W(y)/2}^p \Pi (\dd y) \geq a_g(s) > 0 .
    \end{equation}
    In particular $f_s$ is not identically zero, so $0 < \norm{f_s}_{\Lip} \le \norm{W}_{\Lip}$.
    By Markov's inequality, the definition $H=W^p\cdot g\circ (2W)$ and Condition~(2) in Lemma~\ref{lemma:lower bound Wasserstein convergence} yield
    \begin{equation}
    \label{eq:Wasserstein_growth}
    \int_\XX f_s^p \dd \mu_t
    \leq \E_x \bracket*{W(\xi_t)^p\1{g\circ (2W)(\xi_t)\geq s}}
    \leq  \frac{1}{s} \E_x \bracket*{H(\xi_t)}
    \leq \frac{v(x,t)}{s}.
    \end{equation}

    Let $s = s(x,t)$ be the unique solution of $A_g(s) = 2^{p \vee 1} v(x,t)$, which exists because $A_g$ is strictly increasing with $\lim_{\rho \uparrow \infty} A_g(\rho) = \infty$.
    Thus $s = \paren*{A_g^{-1} \circ (2^{p \vee 1} v)}(x,t)$ and, as $A_g(s) = s a_g(s)$,
    \begin{equation} \label{eq:choice of s}
        \frac{v(x,t)}{s} = \frac{a_g(s)}{2^{p \vee 1}}.
    \end{equation}
    We then apply Proposition~\ref{prop:upper and lower bounds wasserstein distance}.
    Together with $\norm{f_s}_{\Lip} \le \norm{W}_{\Lip}$, \eqref{eq:Wasserstein_bound}, \eqref{eq:Wasserstein_growth} and \eqref{eq:choice of s} gives
    $$
    \norm{W}_{\Lip}^{p \wedge 1} \WW_p(\Pi,\mu_t)
    \ge
    \abs*{\paren*{\int_\XX f_s^p \dd \Pi}^{\frac{1}{p \vee 1}} - \paren*{\int_\XX f_s^p \dd \mu_t}^{\frac{1}{p \vee 1}}
    }
    \ge
    a_g(s)^{\frac{1}{p \vee 1}} - \paren*{\frac{v(x,t)}{s}}^{\frac{1}{p \vee 1}}
    =
    \frac{a_g(s)^{\frac{1}{p \vee 1}}}{2}.
    $$
    Raising both sides to the power $p \vee 1$ and substituting $s = \paren*{A_g^{-1} \circ (2^{p \vee 1} v)}(x,t)$, we obtain \eqref{eq:Wasserstein convergence rate lower bound} as claimed. 
    This completes the proof. 
\end{proof}

\subsection{Proof of ergodicity} \label{subsec:proof ergodicity}

In this subsection we assemble the ingredients of the proof of Theorem~\ref{thm:TLD convergence rates}. 
As discussed in Section~\ref{subsec:intro discussion} and Remark~\ref{remark:convergence rate results}, the upper bounds in $f$-variation are established in \cite[Thm~16]{FR2005}, but under a different, though related, assumption on $\pi$ and for $f$-variations close to $f_p$.
Moreover, \cite[Thm~3.1]{BM2026} establishes the lower bound in \eqref{bound:f var less than lstar} for $d \ge 2$.
Proposition~\ref{prop:f var convergence rate} below states these bounds in the form we need, and its proof is deferred to Appendix~\ref{appendix:upper bound f var convergence rate}.
\begin{proposition} \label{prop:f var convergence rate}
    Let $f_p (x) = 1 + \abs{x}^p$ for $x \in \R^d$, $p \ge 0$. 
    \vspace{3pt}
    \begin{enumerate}
        \item If $\ell  \in \left[0, \lstar\right)$, fix $p \in [0,r)$ and set $\beta_p := (r-p)/(2(1 - \ell (d+r))) > 0$. 
        For any $\eps > 0$ and $x \in \R^d$, there exist constants $C_{-}(x), C_{+}(x) > 0$ such that, for all $t \ge 1$,  
        \begin{equation} \label{bound:f var less than lstar}
            C_{-}(x) t^{- \beta_p - \eps} 
            \le 
            \norm{\P_x (X_t \in \cdot) - \Pi(\cdot)}_{f_p} 
            \le 
            C_{+}(x) t^{- \beta_p + \eps}. 
        \end{equation}
        \item If $\ell = \lstar$, for every $p \in [0, r)$ there exists a constant $c >0$ such that, for every $x \in \R^d$, there is a constant $C(x)>0$ satisfying
        \begin{equation} \label{bound:f var at lstar}
            \norm{\P_x (X_t \in \cdot) - \Pi(\cdot)}_{f_p} 
            \le 
            C(x) e^{- c t} \quad \text{for all} \quad t \ge 1.
        \end{equation}
        \item If $\ell \in (\lstar, \lstarstar]$, there exist constants $c, C>0$ such that, for all $x \in \R^d$,
        \begin{equation} \label{bound:f var greater than lstar TV}
            \norm{\P_x (X_t \in \cdot) - \Pi(\cdot)}_{\TV} 
            \le 
            C e^{- c t} \quad \text{for all} \quad t \ge 1.
        \end{equation}
    \end{enumerate}
\end{proposition}

At $\ell = \lstar$, we deduce the failure of uniform convergence from the non-confinement of~$X$: $\P_x \paren*{\limsup_{t \to \infty} \abs{X_t} = \infty} = 1$ for all $x\in\R^d$; see Lemma~\ref{lemma:unbounded X} in Appendix~\ref{appendix:X is unbounded} below.

Since Proposition~\ref{prop:f var convergence rate}~(3) provides convergence in total variation, our last ingredient upgrades such bounds to the $f_p$-variation. 
The following lemma improves exponential (resp.\ uniform) convergence in total variation to the corresponding mode of convergence in $f_p$-variation, at the cost of a $p$-th moment bound uniform in time (resp.\ in time and space).

\begin{lemma} \label{lemma:TV to fp upgrade}
    Let $\xi = (\xi_t)_{t \ge 0}$ be a Markov process on $\R^d$ with invariant probability measure~$\Pi$, and assume that there exist a constant $c > 0$ and, for each $x \in \R^d$, a constant $C(x) > 0$ such that
    \begin{equation} \label{eq:known TV convergence}
        \norm{\P_x (\xi_t \in \cdot) - \Pi(\cdot)}_{\TV}
        \le
        C(x) e^{- c t}
        \myforall t \ge 1.
    \end{equation}
    Let $\pstar > 0$ and $f_p (x) = 1 + \abs{x}^p$ for $p \ge 0$. 
    Assume further that, for every $p \in [0, \pstar)$,
    \begin{equation} \label{eq:uniform moments for upgrade}
        \int f_{p} \dd \Pi < \infty
        \myand
        \sup_{s \ge 1} \ \E_x \bracket*{\abs{\xi_s}^{p}} < \infty
        \myforall x \in \R^d.
    \end{equation}
    Then, for every $p \in (0, \pstar)$, there exist a constant $\tilde{c} \in \paren*{0 , c \paren*{1 - p/\pstar}}$ and constants $\tilde{C}(x) > 0$, $x \in \R^d$, such that
    \begin{equation} \label{eq:upgraded f var exp conv}
        \norm{\P_x (\xi_t \in \cdot) - \Pi(\cdot)}_{f_p}
        \le
        \tilde{C}(x) e^{- \tilde{c} t}
        \myforall t \ge 1.
    \end{equation}
    If, in addition, 
    $$
    \sup_{x \in \R^d} C(x) < \infty
    \myand 
    \sup_{x \in \R^d} \, \sup_{s \ge 1} \, \E_x \bracket*{\abs{\xi_s}^{p}} < \infty 
    \myforall
    p \in (0, \pstar),
    $$ then, for every $p \in (0, \pstar)$, the constants in \eqref{eq:upgraded f var exp conv} can be chosen so that $\tilde{C} := \sup_{x \in \R^d} \tilde{C}(x) < \infty$.
\end{lemma}

\begin{proof}
    Fix $p \in (0, \pstar)$ and choose $\tilde p \in (p, \pstar)$.
    Let $x \in \R^d$ and $t \ge 1$. 
    Let $g : \R^d \to \R$ be measurable with $\abs{g} \le f_p$, and let $R \ge 1$. 
    Since $f_{p} \le 2 R^{p - \tilde{p}} f_{\tilde{p}}$ on $\B_R^c$, we obtain 
    $$
    \abs*{\int_{\R^d} g(y) \1{y \in \B_R^c} \, \P_x(\xi_t \in \dd y)}
    \le 
    \int_{\R^d} f_p (y) \1{y \in \B_R^c} \, \P_x(\xi_t \in \dd y)
    \le 
    2 \E_x \bracket*{f_{\tilde{p}} (\xi_t)} R^{p - \tilde{p}}, 
    $$
    and, by the same argument, $\abs*{\int_{\R^d} g(y) \1{y \in \B_R^c} \, \Pi(\dd y)} \le 2 \paren*{\int f_{\tilde{p}} \, \dd \Pi} R^{p - \tilde{p}}$. 
    In the ball $\B_R$, we bound $\abs{g}$ by $1 + R^p$, which yields
    \begin{align*}
        \abs*{\int_{\R^d} g(y) \1{y \in \B_R} \, \paren*{\P_x(\xi_t \in \dd y) - \Pi(\dd y)}}
        &= (1 + R^p) \, \abs*{\int_{\R^d} \frac{g(y) \1{y \in \B_R}}{1 + R^p} \, \paren*{\P_x(\xi_t \in \dd y) - \Pi(\dd y)}} \\
        &\le 2 R^p \, \norm{\P_x (\xi_t \in \cdot) - \Pi(\cdot)}_{\TV}.
    \end{align*}
    Combining these three estimates with \eqref{eq:known TV convergence} and taking the supremum over all such $g$, we obtain
    \begin{align} \label{eq:bound on fp by moments}
        \norm{\P_x (\xi_t \in \cdot) - \Pi(\cdot)}_{f_p}
        \le 2C(x) R^p  e^{-c t} + 2 \bracket*{\E_x \bracket*{f_{\tilde{p}} (\xi_t)} + \int f_{\tilde{p}} \, \dd \Pi} R^{p - \tilde{p}}.
    \end{align}
    Choosing $R = R(t) = \exp \paren*{ct / \tilde{p}} \ge 1$ balances the two terms on the right-hand side of \eqref{eq:bound on fp by moments}, in the sense that $R(t)^p e^{-ct} = R(t)^{p - \tilde p}$. 
    Set 
    $$
    \tilde c = c \paren*{1 - \frac{p}{\tilde p}} \in \paren*{0 , c\paren*{1 - \frac{p}{\pstar}}}
    \myand 
    \tilde C (x) = 4 \max \set*{C(x), \, \sup_{s \ge 1} \E_x \bracket*{f_{\tilde{p}} (\xi_s)} + \int f_{\tilde{p}} \, \dd \Pi},
    $$ 
    where $\tilde C (x) < \infty$ by \eqref{eq:uniform moments for upgrade}.  
    Consequently, \eqref{eq:bound on fp by moments} gives
    \begin{align*}
        \norm{\P_x (\xi_t \in \cdot) - \Pi(\cdot)}_{f_p}
        \le \frac{\tilde C (x)}{2} \paren*{R(t)^p e^{-ct} + R(t)^{p - \tilde p}}
        = \tilde{C}(x) e^{- \tilde c t}.
    \end{align*}
    Since $\tilde{C} = \sup_{x \in \R^d} \tilde{C}(x) < \infty$ when $\sup_{x \in \R^d} C(x) < \infty$ and $\sup_{x \in \R^d} \sup_{s \ge 1} \E_x \bracket*{\abs{\xi_s}^{\tilde p}} < \infty$, the proof is now complete. 
\end{proof}

We are now ready to prove Theorem~\ref{thm:TLD convergence rates}. 

\begin{proof}[Proof of Theorem~\ref{thm:TLD convergence rates}]
    For $\ell \in [0, \lstar]$, a direct application of \eqref{eq:f-var control wasserstein} and Proposition~\ref{prop:f var convergence rate} yields the upper bounds in \eqref{ergodicity:polynomial Wasserstein} and \eqref{ergodic:exponentially Wasserstein}. 

    \smallskip
    \noindent
    \underline{$\ell \in [0,\lstar)$.}  
    Fix $p \in (0,r)$.
    For the lower bound, take $\gamma = r - \delta + 2 \abs{\alpha}$ with $\delta \in (0, r \wedge 2 \abs{\alpha})$, so that $0 < r-\delta < \gamma$ and $\gamma > r$. 
    Let $W: \R^d \to [1, \infty)$ be a function satisfying $W \in C^2 (\R^d)$ and $W(x) = \abs{x}$ for $x \in \bar{\B}_2^c$, which makes $W$ Lipschitz. 
    Define $g(u) = u^{\gamma - p}$ for $u \ge 1$ and $H = W^p (2 W)^{\gamma - p} = 2^{\gamma - p} W^{\gamma}$.

    By \eqref{eq:generator on polynomial} and \eqref{eq:generator poly error terms}, there exists $K > 1$ such that $\AA H \le K$ on $\R^d$.
    Define $\xi \equiv K$ and $\Xi (u) = \int_1^u \xi(s)^{-1} \dd s = (u - 1)/ K$ for $u \ge 1$.
    Since $\AA H \le \xi \circ H$, \cite[Lemma~2.10]{BM2026} gives
    \begin{equation} \label{bound:marginal finite p moments}
        \E_x \bracket*{H(X_t)}
        \le \Xi^{-1} \paren*{\Xi (H(x)) + t}
        = H(x) + K t
        =: v(x,t).
    \end{equation}
    By \eqref{AP}, there exists a constant $C \in (0,1)$ such that for any $\rho \ge 1$
    $$
    \int_{\set*{g \circ W \ge \rho}} \paren*{W(x)/2}^p \, \Pi(\dd x)
    = \int_{\set*{W \ge \rho^{{1}/\paren*{\gamma - p}}}} \paren*{W(x)/2}^p \, \Pi(\dd x)
    \ge C \rho^{- \frac{r - p}{\gamma - p}}.
    $$
    Define $a_g (\rho) = C \rho^{- \frac{r - p}{\gamma - p}}$ and $A_g(\rho) = \rho a_g(\rho) = C \rho^{\frac{\gamma - r}{\gamma - p}}$ for $\rho \ge 1$. 
    As $\gamma > r > p$, $A_g$ is strictly increasing with $\lim_{\rho \uparrow \infty} A_g (\rho) = \infty$. 
    By \eqref{AP} and \eqref{bound:marginal finite p moments}, we have $\Pi, \P_x(X_t \in \cdot) \in \PP_p(\R^d)$. 
    Lemma~\ref{lemma:lower bound Wasserstein convergence} then yields, for all $x \in \R^d$ and $t \ge 1$,
    $$
    \WW_p(\Pi,\P_x (X_t \in \cdot))^{p \vee 1}
    \ge \frac{\paren*{a_g \circ A_g^{-1}\circ (2^{p \vee 1} v)} (x,t)}{2^{p \vee 1} \norm{W}_{\Lip}^p}
    \ge C_{-}(x) t^{-\frac{r-p}{2 \abs{\alpha}} - \delta''}
    $$
    where $ C_{-}(x) >0$ and $\delta'' = \frac{r-p}{2 \abs{\alpha} - \delta} - \frac{r-p}{2 \abs{\alpha}}$.
    This completes the justification of \eqref{ergodicity:polynomial Wasserstein}. 

    \smallskip
    \noindent
    \underline{$\ell = \lstar$.}
    Choose $V: \R^d \to [1, \infty)$ with $V \in C^2(\R^d)$ such that $V(x) = \log \abs{x}$ for $x \in \bar{\B}_3^c$. 
    For $\ell = \lstar$, \eqref{eq:generator on general functions} gives
    $$
    \AA (1/V) (x) = \frac{\theta_\ell (x)^2}{(\log \abs{x})^2} \bracket*{r - \eps_\ell (x) + \frac{2}{\log \abs{x}}}
    \for
    x \in \bar{\B}_3^c.
    $$
    Since $\AA (1/V)$ is continuous, \eqref{eq:generator poly error terms} shows that there exists $K>0$ with $\AA (1/V) \le K (1/V)^2$ on~$\R^d$.
    Setting $\varphi(u) = K u^2$ for $u \in (0,1]$, \cite[Thm~2.7 (a)]{BM2026} then yields, for every $x \in \R^d$, 
    $$
    \paren*{{1}/{V(X_t)} - \int_0^t \varphi \paren*{{1}/{V(X_s)}} \dd s}_{t \ge 0}
    \quad
    \text{is a $(\FF_t)_{t \ge 0}$-supermartingale under $\P_x$.}
    $$
    By Lemma~\ref{lemma:unbounded X}, we also have $\limsup_{t \to \infty} V(X_t) = \infty$ $\P_x$-almost surely.

    Fix $q \in (1/2, 1)$ and choose $R_q > \max \set*{3, \sup_{y \in \bar{\B}_3} V(y)}$ with $\Pi \paren*{\B_{e^{R_q}}} \ge q$.
    For $t > 0$, set $\rho_{q,t} = \frac{2 K t + R_q}{1-q}$, so that $\rho_{q,t} > R_q$ and $t < \tilde t := \frac{1-q}{2 \rho_{q,t} \, \varphi \paren{1/ \rho_{q,t} }} = t + \frac{R_q}{2K}$.
    Then \cite[Lemma~4.2]{BM2026} gives that
    $$
    \P_x \paren*{\tau_{R_q} \ge \tilde t \,} \ge q,
    \mywhere 
    \tau_{R_q} = \inf \set{s \ge 0: V(X_s) < R_q},
    $$
    for any $x \in \R^d$ with $\abs{x} = \exp \paren*{2\rho_{q,t}/(1-q)}$. 
    As $R_q > \sup_{y \in \bar{\B}_3} V(y)$, we have 
    $$
    q
    \le
    \P_x \big( \inf_{s \le \tilde t} V(X_s) \ge R_q \big)
    \le 
    \P_x \big(\inf_{s \le t} \, \log \abs{X_s} \ge R_q \big)
    \le
    \P_x \paren*{\abs{X_t} \ge e^{R_q}},
    $$
    and consequently,
    $$
    \norm{\P_x \paren*{X_t \in \cdot} - \Pi}_{\TV}
    \ge
    2 \paren*{\Pi \paren*{\B_{e^{R_q}}} - \P_x \paren*{\abs{X_t} < e^{R_q}}}
    \ge
    2(2q - 1). 
    $$
    Letting $q \uparrow 1$ then gives $\sup_{x \in \R^d} \norm{\P_x (X_t \in \cdot) - \Pi(\cdot)}_{\TV} = 2$ for all $t > 0$, which completes the proof of Theorem~\ref{thm:TLD convergence rates} (2). 

    \smallskip
    \noindent
    \underline{$\ell \in (\lstar, \lstarstar]$.} 
    Under \eqref{AP}, $\int f_p \dd \Pi < \infty$ for every $p \in [0,r)$, and Theorem~\ref{thm:TLD moments} gives the bound \eqref{eq:moments uniform in time and space above lstar} on the $p$-th moments, uniform in space and time. 
    Lemma~\ref{lemma:TV to fp upgrade} extends \eqref{bound:f var greater than lstar TV} in Proposition~\ref{prop:f var convergence rate} to uniform convergence in the $f_p$-variation distance for all $p \in [0,r)$ which establishes \eqref{ergodic:uniformly f var}.
    The upper bounds on ${\WW_p (\P_x (X_t \in \cdot), \Pi)^{p \vee 1}}$ again follow directly from \eqref{eq:f-var control wasserstein}. 
    This proves Theorem~\ref{thm:TLD convergence rates}~(3). 
\end{proof}

\appendix

\section{Proofs of auxiliary results} 

\subsection{Negative moments of Bessel processes} \label{appendix:proof negative moments of Bessel processes}
We now present the properties of the closed-form expression for the negative moments of Bessel processes in Lemma~\ref{lemma:properties of Bessel neg moments}.
\begin{proof}[Proof of Lemma~\ref{lemma:properties of Bessel neg moments}]
    Fix $s > 0$ and $z \ge 0$, and set $\lambda = z^2/s$, so that
    $$
    \frac{U_{\delta, \eta} (s,z)}{\paren*{s+z^2}^{-\eta/2}}
    =
    \frac{K_{\delta, \eta} M \paren*{\frac{\eta}{2}, \frac{\delta}{2}, -\frac{\lambda}{2}}}{(1+\lambda)^{-\eta/2}}
    =: G(\lambda).
    $$
    Kummer's transformation \cite[(13.1.27)]{AbramowitzStegun1972} gives $M \paren*{\frac{\eta}{2}, \frac{\delta}{2}, -\frac{\lambda}{2}} = e^{-\lambda/2} M \paren*{\frac{\delta - \eta}{2}, \frac{\delta}{2}, \frac{\lambda}{2}} >0$ for all $\lambda \ge 0$.
    Moreover, since $0<\eta < \delta$, \cite[(13.1.5)]{AbramowitzStegun1972} yields
    $$
    M \paren*{\frac{\eta}{2}, \frac{\delta}{2}, -\frac{\lambda}{2}}
    \sim
    \frac{\Gamma \paren*{\delta/2}}{\Gamma \paren*{\paren{\delta - \eta}/{2}}} \paren*{\frac{\lambda}{2}}^{-\eta/2}
    = K_{\delta, \eta}^{-1} \lambda^{-\eta/2}
    \as
    \lambda \to \infty,
    $$
    so that $G(\lambda) \to 1$ as $\lambda \to \infty$.
    Since $G: [0, \infty) \to (0, \infty)$ is continuous with $\lim_{\lambda \to \infty} G(\lambda) = 1$, it follows that $c_{\delta, \eta} := \inf_{\lambda \ge 0} G(\lambda) > 0$ and $C_{\delta, \eta} := \sup_{\lambda \ge 0} G(\lambda) < \infty$, establishing \eqref{eq:Bessel neg moments two sided bounds}.

    The identities \eqref{eq:Bessel neg moments derivatives} follow from $\frac{\dd}{\dd z} M(a,b,z) = \frac{a}{b} M(a+1, b+1, z)$ \cite[(13.4.8)]{AbramowitzStegun1972}, and substituting them into \eqref{eq:Bessel neg moments PDE} reduces it to Kummer's equation \cite[(13.1.1)]{AbramowitzStegun1972}.

    For $s,z > 0$, the first two parts of the lemma give
    \begin{align*}
        \abs*{\partial_s U_{\delta, \eta}(s,z)} 
        &\le \frac{\eta \delta}{2} U_{\delta+2, \eta+2} (s,z)
        + \frac{\eta (\eta+2)}{2} z^2 U_{\delta+4, \eta+4} (s,z) \\
        &\le \frac{\eta \delta}{2} C_{\delta+2, \eta+2} \paren*{s+z^2}^{-(\eta+2)/2} + \frac{\eta (\eta+2)}{2} C_{\delta+4, \eta+4} z^2 \paren*{s+z^2}^{-(\eta+4)/2} \\
        &\le \paren*{\frac{\eta \delta}{2} C_{\delta+2, \eta+2} + \frac{\eta (\eta+2)}{2} C_{\delta+4, \eta+4}} \paren*{s+z^2}^{-(\eta+2)/2}
    \end{align*}
    and 
    $$
    \frac{\abs*{\partial_z U_{\delta, \eta} (s,z)}}{z}
    = \eta U_{\delta+2, \eta+2} (s,z)
    \ge \eta c_{\delta+2, \eta+2} \paren*{s+z^2}^{-(\eta+2)/2},
    $$
    which completes the proof. 
\end{proof}

\subsection{Integral representation of the \texorpdfstring{$f$}{f}-variation} \label{appendix:proof f-variation representation}
We prove Lemma~\ref{lemma:bound by f-var}, which expresses the $f$-variation of a signed measure $\eta$ as the integral of $f$ against its total variation measure $\abs{\eta}$.
\begin{proof}[Proof of Lemma~\ref{lemma:bound by f-var}]
    By the Hahn decomposition theorem, there exist two subsets $A_1, A_2 \in \EE$ such that $A_1 \cup A_2 = E$ and $A_1 \cap A_2 = \emptyset$. 
    Moreover, $\eta^{+} (\cdot) := \eta(A_1 \cap \cdot)$ and $\eta^{-} (\cdot) := -\eta(A_2 \cap \cdot)$ are two non-negative measures on $(E, \EE)$. 
    Let $\tilde{f} = f \1{A_1} - f \1{A_2}$. 
    As $\abs{\tilde{f}} = f$, we have 
    \begin{align*}
        \norm{\eta}_f
        \ge \abs*{\int \tilde{f} \dd \eta}
        = \abs*{\int f \1{A_1} \dd \eta - \int f \1{A_2} \dd \eta}
        = \int f \dd \eta^+ + \int f \dd \eta^-
        = \int f \dd \abs{\eta}.
    \end{align*}
    Conversely, $\abs{\int g \dd \eta} \le \int \abs{g} \dd \abs{\eta} \le \int f \dd \abs{\eta}$ for every measurable $g$ with $\abs{g} \le f$, so $\norm{\eta}_f = \int f \dd \abs{\eta}$.
\end{proof}

\subsection{Convergence rates in \texorpdfstring{$f$}{f}-variation distances} \label{appendix:upper bound f var convergence rate}
The proof of Theorem~\ref{thm:TLD convergence rates} relies on Foster–Lyapunov type drift conditions, as presented in \cite{DoucSubgeometric09, dmt1995exponential,BM2026}.
Applying these drift conditions relies crucially on the concept of petite sets.
Following the definition in \cite[p.~500]{mt1993stability2}, a set $C \in \BB (\R^d)$ is called a \textit{petite set} if there exists a probability measure $a$ on $(\R_+, \BB (\R_+))$ such that 
$$
\inf_{x \in C} \int_0^\infty \P_x (X_t \in \cdot) a(\dd t) \ge \eta_C (\cdot)
$$
where $\eta_C$ is a non-trivial $\sigma$-finite measure on $(\R^d, \BB (\R^d))$. 
\begin{proposition} \label{prop:TLD MC regularity}
    Any skeleton chain associated with $X$ is irreducible, and every compact set is a closed petite set for $X$. 
\end{proposition}

\begin{proof}
    As in the proof of \cite[Prop.~15]{FR2005}, we conclude the irreducibility of the skeleton chains from \eqref{eq:convergence to inv measure}. 
    The petiteness of compact subsets comes from the continuity of $p(t, \cdot, \cdot)$. 
\end{proof}

We next introduce two further concepts needed for the proof of exponential ergodicity (see, e.g., \cite[pp.~490--491]{mt1993stability2}).
We say that $X$ is
\begin{itemize}
    \item \textit{$\psi$-irreducible} if $\psi$ is a $\sigma$-finite measure on $(\R^d, \BB (\R^d))$ such that for any $A \in \BB (\R^d)$ with $\psi(A) > 0$, we have $\int_0^\infty \P_x (X_t \in A) \dd t > 0$ for all $x \in \R^d$;
    \item \textit{positive Harris recurrent} if $X$ has a finite invariant measure and there exists a non-zero $\sigma$-finite measure $\varphi$ on $(\R^d, \BB (\R^d))$ such that for any $A \in \BB (\R^d)$ with $\varphi(A) > 0$, we have $\P_x (\int_0^\infty \1{X_t \in A} \dd t = \infty)=1$ for all $x \in \R^d$.
\end{itemize}

\begin{proof}[Proof of Proposition~\ref{prop:f var convergence rate}]
    We apply the Foster–Lyapunov type drift conditions in \cite{DoucSubgeometric09, BM2026, dmt1995exponential}.
    
    \smallskip
    \noindent \underline{Case $\ell \in [0,\lstar)$}  
    Fix $p \in [0,r)$, choose $\delta \in \paren*{0, (r-p) \wedge 2\abs{\alpha}}$, and set $\gamma = r - \delta + 2\abs{\alpha}$ and $\beta = \frac{p}{r-\delta} \in [0,1)$. 
    We further assume that $\VV_\gamma(x) \le 1 + \abs{x}^{\gamma}$ for all $x \in \R^d$.

    For the upper bound, by \eqref{eq:generator on polynomial} and \eqref{eq:generator poly error terms}, there exist $R > 2$ and $\lambda > 0$ such that $\AA \VV_\gamma \le - \lambda \VV_\gamma^{\paren{r - \delta}/\gamma}$ on $\bar\B_R^c$.
    With $\phi (u) = \lambda u^{\paren{r - \delta}/\gamma}$ and $K = 1+ \sup_{x \in \bar{\B}_R} \abs{\AA \VV_\gamma(x) + \phi \circ \VV_\gamma(x)} < \infty$, we obtain 
    $$
    \AA \VV_\gamma \le - \phi \circ \VV_\gamma + K \1{\bar{\B}_R} \myon \R^d.
    $$ 
    Since $\bar{\B}_R$ is a closed petite set by Proposition~\ref{prop:TLD MC regularity}, this establishes the drift condition $\mathbf{D}(\bar{\B}_R, \VV_\gamma, \phi, K)$ of \cite[p.~901]{DoucSubgeometric09}, via \cite[Thm~3.4]{DoucSubgeometric09}.
    Defining
    $$
    \Psi_1 (u) = 2^{-(p+1)} \lambda^{\beta} u^{1-\beta}
    \myand
    \Psi_2 (u) = 2^{p+1} \lambda^{-\beta} u^{\beta}
    \for u \ge 0,
    $$
    one checks that $(\Psi_1, \Psi_2)$ forms a pair of inverse Young functions in the sense of \cite[Sec.~3.1, (3.3)]{DoucSubgeometric09}.
    Let $\varphi = \phi \, \circ \, \VV_\gamma$. 
    One computes that 
    \begin{equation} \label{eq:bound fp}
        \Psi_2 \circ \varphi (x) = 2^{p+1} \VV_\gamma(x)^{p/\gamma}
        \ge
        2^{p+1} \1{x \in \bar{\B}_2} + 2^{p+1} \abs{x}^p \1{x \in \bar{\B}_2^c}
        \ge
        f_p(x)
        \myforall
        x \in \R^d.
    \end{equation}
    Since every skeleton chain of $X$ is irreducible and $\lim_{u \to \infty} \phi'(u) = 0$ and \eqref{eq:bound fp} holds, \cite[Thm~3.2]{DoucSubgeometric09} yields that there exists a constant $C_+(x) >0$ for each $x \in \R^d$ such that, for all $t \ge 1$,
    $$
    \norm{\P_x (X_t \in \cdot) - \Pi(\cdot)}_{f_p}
    \le
    \norm{\P_x (X_t \in \cdot) - \Pi(\cdot)}_{\Psi_2 (\varphi) \vee 1}
    \le
    C_+(x) t^{-\frac{r-p}{2 \abs{\alpha}} + \frac{\delta}{2 \abs{\alpha}}}.
    $$
    
    For the lower bound, let $h= \id$ be the identity function on $[1, \infty)$.
    Define $\xi \equiv K \ge 1$ and $\Xi (u) = \int_1^u \xi(s)^{-1} \dd s = (u - 1)/K$ for $u \ge 1$.
    Since $\AA \VV_\gamma \le \xi \circ \VV_\gamma$, by \cite[Lemma~2.10]{BM2026}, we obtain
    $$
    \E_x \bracket*{h \circ \VV_\gamma(X_t)} \le \Xi^{-1} \paren*{\Xi (\VV_\gamma(x)) + t} = \VV_\gamma(x) + K t.
    $$
    We define $f_\star (u) = u^{p/ \gamma}$ and $g (u) = h(u)/f_\star (u) = u^{(\gamma - p)/ \gamma}$ where $u \ge 1$.
    By \eqref{AP}, there exists a constant $\kappa > 0$ such that the tail of $\Pi$ (with respect to $\VV_\gamma$) satisfies
    $$
    \Pi (\VV_\gamma > \rho)
    \ge \kappa \rho^{- r /\gamma}
    =: L(\rho), 
    \quad \rho \ge 1. 
    $$
    Define, for $\rho \ge 1$, 
    $$ 
    a_g(\rho) = f_\star \paren*{g^{-1} (\rho)} L \paren*{g^{-1}(\rho)} = \kappa \rho^{- \frac{r-p}{\gamma - p}} 
    \myand
    A_g(\rho) = \rho a_g(\rho) = \kappa \rho^{\frac{2 \abs{\alpha} - \delta}{(r - p) + 2 \abs{\alpha} - \delta}}.
    $$  
    Since $\delta < 2 \abs{\alpha}$, $A_g$ is a strictly increasing function and $A_g(\rho) \uparrow \infty$ as $\rho \uparrow \infty$. 
    Applying \cite[Lemma~5.1]{BM2026}, we conclude that, for all $x \in \R^d$ and $t \ge 1$,
    \begin{align*}
              \norm{\P_x (X_t \in \cdot) - \Pi(\cdot)}_{f_\star\circ\mathcal V_\gamma}
        \ge \frac{1}{2} (a_g \circ A_g^{-1}) (2(\VV_\gamma(x) + K t))
        \ge C_{-}(x) t^{-\frac{r-p}{2 \abs{\alpha}} - \delta''}
    \end{align*}
    where $ C_{-}(x) >0$ and $\delta'' = \frac{r-p}{2 \abs{\alpha} - \delta} - \frac{r-p}{2 \abs{\alpha}}$.
    As $p/ \gamma \in [0,1)$ and $\VV_\gamma(x) \le 1 + \abs{x}^{\gamma}$ for all $x \in \R^d$, together with \eqref{fact:subadditivity}, we have 
    $$
    f_\star \circ \VV_\gamma (x) \le (1+\abs{x}^\gamma)^{p/\gamma} \le 1+ \abs{x}^p = f_p(x)
    $$
    and thus $\norm{\cdot}_{f_p} \ge \norm{\cdot}_{f_\star\circ\mathcal V_\gamma}$, which completes the proof of \eqref{bound:f var less than lstar}.

    \smallskip
    \noindent \underline{Case $\ell = \lstar$}
    Let $p \in (0,r)$ and $V_p: \R^d \to [1, \infty)$ be a function with $V_p \in C^2(\R^d)$, $V_p(x) = \abs{x}^p + 2$ on $\bar\B_2^c$, and $V_p \ge f_p$.
    Combining \eqref{eq:generator on polynomial} and \eqref{eq:generator poly error terms} shows that there exist $\lambda >0$ and $R>2$ large enough that $\AA V_p \le - \lambda V_p$ on $\bar{\B}_R^c$. 
    Letting $K = \sup_{x \in \bar{\B}_R} \abs{\AA V_p(x) + \lambda V_p(x)} < \infty$ gives $\AA V_p \le - \lambda V_p + K \1{\bar{\B}_R}$ on $\R^d$. 
    As in the case $\ell \in [0, \lstar)$, $\bar{\B}_R$ is a closed petite set by Proposition~\ref{prop:TLD MC regularity}. 

    Applying \cite[Thm~5.2 (c)]{dmt1995exponential} requires verifying that $X$ is \textit{aperiodic} in the sense of \cite[p.~1675]{dmt1995exponential}.
    Combining \cite[Thm~4.2]{mt1993stability3} with $\AA V_p \le - \lambda V_p + K \1{\bar{\B}_R}$ on $\R^d$ shows that $X$ is positive Harris recurrent.
    Since every skeleton chain of $X$ is irreducible, \cite[Prop.~6.1]{mt1993stability2} then gives a petite set ${A} \in \BB(\R^d)$, a constant $T>0$, and a non-trivial measure $\eta_{A}$ with $\eta_{A} ({A}) > 0$ such that
    $
    \P_x (X_t \in \cdot) \ge \eta_{A} (\cdot),
    $ 
    for all $x \in {A}$ and $t \ge T$,
    and in particular $\P_x (X_t \in {A}) > 0$ for all $x \in {A}$ and $t \ge T$.
    As $X$ is $\Pi$-irreducible and aperiodic, \cite[Thm~5.2 (c)]{dmt1995exponential} yields constants $C, c>0$ such that
    \begin{equation*}
        \norm{\P_x (X_t \in \cdot) - \Pi(\cdot)}_{V_p} \le C V_p(x) e^{-c t} \myforall x \in \R^d \myand t \ge 0
    \end{equation*}
    which proves~\eqref{bound:f var at lstar} for $p\in(0,r)$ because $f_p \le V_p$. The case $p=0$ follows by dominated convergence with some $f_q$ for $q\in(0,r)$.

    \smallskip
    \noindent \underline{Case $\ell \in (\lstar, \lstarstar]$}  
    In this case, $\alpha = \ell (d+r) - 1 \in (0, r/2]$.
    Let $\Phi: \R^d \to (0, \infty)$ be a function with $\Phi \in C^2(\R^d)$ and $\Phi (x) = \frac{1}{\sqrt{2} \alpha} \abs{x}^{-\alpha}$ for $x \in \bar{\B}_2^c$, and set
    $$
    M = \sup_{x \in \R^d} \Phi(x)^2 + 1 < \infty
    \myand
    V(x) = M - \Phi(x)^2 \for x \in \R^d.
    $$
    By \eqref{eq:generator on polynomial} and \eqref{eq:generator poly error terms}, there exist $\lambda >0$ and $R>2$ large enough that $\AA V \le - \lambda V$ on $\bar{\B}_R^c$. 
    Letting $K = \sup_{x \in \bar{\B}_R} \abs{\AA V(x) + \lambda V(x)} < \infty$, we obtain $\AA V \le - \lambda V + K \1{\bar{\B}_R}$ on $\R^d$. 
    As in the case $\ell = \lstar$, applying \cite[Thm~5.2 (c)]{dmt1995exponential} then yields constants $C', c > 0$ with $\norm{\P_x (X_t \in \cdot) - \Pi(\cdot)}_{V} \le C' V(x) e^{-c t}$.
    Since $1 \le V \le M$ on $\R^d$, taking $C = C' M$ gives
    \begin{equation} \label{eq:uniform in TV above lstar}
        \norm{\P_x (X_t \in \cdot) - \Pi(\cdot)}_{\TV}
        \le
        C e^{-c t}
        \myforall x \in \R^d \myand t \ge 0. 
    \end{equation}
    This proves \eqref{bound:f var greater than lstar TV} and completes the proof of this proposition. 
\end{proof}

\subsection{Non-confinement of tempered Langevin diffusions} \label{appendix:X is unbounded}
We prove that the trajectories of tempered Langevin diffusions are not containd a in compact subset of $\R^d$ with positive probability. 

\begin{lemma} \label{lemma:unbounded X}
    When $\ell \in [0,\lstarstar]$, $\limsup_{t \to \infty} \abs{X_t} = \infty$ $\P_x$-almost surely for every $x \in \R^d$. 
\end{lemma}

\begin{proof}[Proof of Lemma~\ref{lemma:unbounded X}]
    Fix $M \in \N$ and $T > 0$. 
    When $X$ is non-explosive, $X$ is Feller and its transition density is continuous and strictly positive by Proposition~\ref{prop:TLD density}. 
    Thus, $\delta_M := \inf_{\abs{y} \le M} \P_y (\abs{X_T} > M) > 0$. 
    
    Let $x \in \R^d$. 
    Define $E_n = \set{\abs{X_{nT}} > M}$ and $\GG_n = \FF_{n T}$ for $n \in \Z_{\ge 0}$. 
    On $E_{n-1}^c = \set{\abs{X_{(n-1)T}} \le M}$, applying the Markov property gives $\P_x (E_n \vert \GG_{n-1}) = \P_{X_{(n-1)T}} (E_1) \ge \delta_M$ up to a $\P_x$-null set. 
    Then, on $\set{E_n \text{ f.o.}}$, the event $E_{n}^c$ holds for all $n$ large enough and thus 
    $$
    \set{E_n \text{ f.o.}} \subset \big\{ \sum_{n=1}^\infty \P_x (E_n \vert \GG_{n-1}) = \infty \big\}
    \quad
    \text{up to a $\P_x$-null set}. 
    $$
    However, by Lévy's extension of the Borel–Cantelli lemma (see e.g.\ \cite[Cor.~9.21]{KallenbergFOP2021}), we have 
    $$
    \big\{\sum_{n=1}^\infty \P_x (E_n \vert \GG_{n-1}) = \infty \big\} = \set{E_n \text{ i.o.}} \quad \P_x\text{-almost surely}.
    $$ 
    Therefore, $\P_x (\set{E_n \text{ f.o.}}) = 0$ and $\P_x \paren*{\limsup_{n \to \infty} \abs{X_{nT}} \ge M} = 1$. 
    Intersecting $\set{\limsup_{n \to \infty} \abs{X_{nT}} \ge M}$ over $M \in \N$ implies $\P_x \paren*{\limsup_{t \to \infty} \abs{X_t} = \infty} = 1$, which completes the proof. 
\end{proof}

\section*{Funding}
AM was also supported  the EPSRC grants EP/V009478/1 and EP/W006227/1. HX gratefully acknowledges funding from the Statistics CDT at the University of Warwick.

\section*{AI statement}
The main ideas underlying the proofs of the principal results were developed by the authors without AI assistance. During the writing process, AI tools (GPT-5.6 Sol and Claude Opus 5.5) were used for checking technical arguments, identifying relevant references, and correcting spelling and language. 

\bibliographystyle{amsalpha-initials}
\bibliography{reference}

\end{document}